\documentclass[oneside,reqno]{amsart}
\usepackage{amsmath, amsthm, mathrsfs, amssymb, amsfonts, mathtools}
\usepackage{graphicx}
\usepackage{hyperref}
\hypersetup{colorlinks=true, linkcolor=blue, citecolor=magenta, filecolor=magenta, urlcolor=cyan}

\usepackage{doi}

\usepackage{tikz}
\usetikzlibrary{patterns.meta}
\tikzset{regions diagram/.style={
x=1cm,
    y=1cm,
    >=stealth,
    line width=0.9pt,
    line cap=round,
    line join=round,
    font=\LARGE
}}

\newcommand{\RegionsGeometry}{
\def\dx{1.35}
\def\dy{0.95}
\def\Lwidth{3.80}
\def\Lheight{1.93}
\pgfmathsetmacro{\slope}{\dy/\dx}
\def\cornerX{1.75}
\pgfmathsetmacro{\cornerY}{\slope*\cornerX}
\def\blobx{2.40}
\def\blobOffset{0.05}
\pgfmathsetmacro{\bloby}{\slope*\blobx+\blobOffset}
\pgfmathsetmacro{\upperintercept}{\Lheight}
\pgfmathsetmacro{\lowerintercept}{
    \cornerY-\slope*(\cornerX+\Lwidth)
}
\pgfmathsetmacro{\lowerxzero}{
    -\lowerintercept/\slope
}
}

\def\hatchAngle{45}
\def\hatchDistance{1.75pt}
\def\hatchLineWidth{0.30pt}
\tikzset{
    left hatch/.style={
        pattern={
            Lines[
                angle=\hatchAngle,
                distance=\hatchDistance,
                line width=\hatchLineWidth
            ]
        },
        pattern color=black
    },
    right hatch/.style={
        pattern={
            Lines[
                angle=-\hatchAngle,
                distance=\hatchDistance,
                line width=\hatchLineWidth
            ]
        },
        pattern color=black
    }
}

\def\blob{
    (-0.48,0.55)
    .. controls (-0.53,0.55) and (-0.57,0.50)
       .. (-0.57,0.43)
    --
    (-0.57,-0.28)
    .. controls (-0.57,-0.38) and (-0.49,-0.44)
       .. (-0.39,-0.44)
    --
    (0.51,-0.44)
    .. controls (0.61,-0.44) and (0.67,-0.37)
       .. (0.67,-0.28)
    --
    (0.67,0.02)
    .. controls (0.67,0.11) and (0.59,0.17)
       .. (0.50,0.17)
    --
    (0.18,0.17)
    .. controls (0.10,0.17) and (0.06,0.22)
       .. (0.06,0.30)
    --
    (0.06,0.44)
    .. controls (0.06,0.51) and (-0.01,0.55)
       .. (-0.09,0.55)
    --
    cycle
}

\def\originblob{
    (0,0)
    --
    (1.18,0)
    --
    (1.18,0.47)
    .. controls (1.18,0.58) and (1.10,0.64)
       .. (1.00,0.64)
    --
    (0.69,0.64)
    .. controls (0.59,0.64) and (0.54,0.70)
       .. (0.54,0.80)
    --
    (0.54,1.17)
    .. controls (0.54,1.27) and (0.47,1.32)
       .. (0.37,1.32)
    --
    (0,1.32)
    --
    (0,0)
    --
    cycle
}

\newcommand{\RegionsBoundaries}{
\draw[->]
    (-0.35,0) -- (12.8,0)
    node[below left] {$m$};
\draw[->]
    (0,-0.35) -- (0,8.2)
    node[above left] {$n$};
\node[below left, inner sep=6pt] at (0,0) {$0$};
\draw
    (0,\upperintercept)
    --
    (7.55,{7.55*\slope+\upperintercept});
\draw[dashed]
    (7.55,{7.55*\slope+\upperintercept})
    --
    (8.85,{8.85*\slope+\upperintercept});
\draw
    (\lowerxzero,0)
    --
    (11.05,{11.05*\slope+\lowerintercept});
\draw[dashed]
    (11.05,{11.05*\slope+\lowerintercept})
    --
    (12.55,{12.55*\slope+\lowerintercept});
\foreach \k in {0,...,4} {
    \pgfmathsetmacro{\lx}{\cornerX+\k*\dx}
    \pgfmathsetmacro{\ly}{\cornerY+\k*\dy}
    \draw
        (\lx,{\ly+\Lheight})
        --
        (\lx,\ly)
        --
        ({\lx+\Lwidth},\ly);
}
\pgfmathsetmacro{\lx}{\cornerX+5*\dx}
\pgfmathsetmacro{\ly}{\cornerY+5*\dy}

\draw[dashed]
    (\lx,{\ly+\Lheight})
    --
    (\lx,\ly)
    --
    ({\lx+\Lwidth},\ly);
}

\newcommand{\RegionsCenterLine}{
\draw
    (0,0)
    --
    (8.35,{8.35*\slope});
\draw[dashed]
    (8.35,{8.35*\slope})
    --
    (11.15,{11.15*\slope});
}

\colorlet{extendedFill}{black!15}

\usepackage[backend=biber, style=trad-alpha, doi=true, isbn=false, hyperref=true, backref=false, giveninits=true, abbreviate=true, sorting=nyt]{biblatex}
\usepackage[T1]{fontenc}
\usepackage[scaled=0.95]{inconsolata}
\usepackage{fvextra}

\DefineVerbatimEnvironment{ShellInput}{Verbatim}{
    fontsize=\small,
    fontfamily=tt,
    fontseries=b,
    frame=none,
    xleftmargin=1.5em,
    breaklines=true
}

\DefineVerbatimEnvironment{ShellOutput}{Verbatim}{
    fontsize=\small,
    fontfamily=tt,
    fontseries=m,
    frame=none,
    xleftmargin=1.5em,
    breaklines=true
}

\DefineVerbatimEnvironment{CodeExample}{Verbatim}{
    fontsize=\small,
    fontfamily=tt,
    fontseries=m,
    frame=none,
    xleftmargin=1.5em,
    breaklines=true
}

\definecolor{verylightgray}{rgb}{0.9,0.9,0.9}
\newcommand{\Kzero}{K}
\newcommand{\setS}{\mathcal{S}}
\newcommand{\Mathematica}{\textit{Mathematica}}
\newcommand{\filename}[1]{\texttt{\detokenize{#1}}}
\newcommand{\gitRepName}{\detokenize{conformal-wave-cap}}

\newtheorem{theorem}{Theorem}

\newtheorem{lemma}{Lemma}

\newtheorem{obs}{Observation}
\theoremstyle{definition}

\theoremstyle{remark}

\usepackage{orcidlink}

\usepackage{afterpage}

\numberwithin{equation}{section}
\mathtoolsset{showonlyrefs}

\begin{document}
\title[]{New time-periodic solutions to the cubic conformal wave equation on $\mathbb{S}^{3}$: I Spherical symmetry}

\author{Filip Ficek\hspace{0.233ex}\orcidlink{0000-0001-5885-7064}}

\author{Maciej Maliborski\hspace{0.233ex}\orcidlink{0000-0002-8621-9761}}

\address{TU Wien, Institute of Analysis and Scientific Computing, Wiedner Hauptstra\ss{}e 8--10, 1040 Vienna, Austria and University of Vienna, Faculty of Mathematics, Oskar-Morgenstern-Platz 1, 1090 Vienna, Austria and University of Vienna, Gravitational Physics, Boltzmanngasse 5, 1090 Vienna, Austria}

\email[Filip Ficek]{filip.ficek@univie.ac.at}

\email[Maciej Maliborski]{maciej.maliborski@tuwien.ac.at}

\thanks{We acknowledge the support of the Austrian Science Fund (FWF) through Projects \href{http://doi.org/10.55776/PAT4519925}{\texttt{10.55776/PAT4519925}}, \href{http://doi.org/10.55776/P36455}{\texttt{10.55776/P36455}}, and \href{http://doi.org/10.55776/PAT9429324}{\texttt{10.55776/PAT9429324}}, and the START-Project \href{http://doi.org/10.55776/Y963}{\texttt{10.55776/Y963}}. Computations were performed on a workstation \texttt{tensor} hosted by the \href{https://www.tuwien.at/en/mg/asc}{Institute of Analysis and Scientific Computing (ASC), TU Wien}}


\begin{abstract}
    We present a rigorous construction of novel time-periodic solutions to the radially symmetric cubic conformal wave equation on the Einstein cylinder $\mathbb{R}\times\mathbb{S}^{3}$. Their existence, hinted by the earlier numerical studies \cite{Ficek.2025a}, complements prior results on Cantor sets of infinitesimally small solutions.
    Our proof, based on the contraction principle, requires extensive computations that are performed on a computer by the means of rational arithmetic. We are able to significantly reduce their amount by a thorough analysis of the inverse of the associated linear operator, leading to sharper analytical bounds, as well as introduction of simplified versions of certain relevant operators.
    We provide examples of three solutions with specific frequency; however, our techniques and tools can be directly adapted to prove the existence of other solutions conjectured in \cite{Ficek.2025a}.
\end{abstract}

\date{\today}

\maketitle

\tableofcontents

\section{Introduction}
\label{sec:Introduction}

\subsection{The model}
\label{sec:Model}

The conformal cubic wave equation on the Einstein cylinder $\mathbb{R}\times\mathbb{S}^{3}$ provides a geometrically natural model of nonlinear wave dynamics under spatial confinement. We study a real scalar field $\phi$ satisfying
\begin{equation}
	\label{eq:ConformalCovariant}
	\Box_{g}\phi-\frac{1}{6}R(g)\phi-\phi^{3}=0\,,
\end{equation}
where $\Box_{g}=g^{\mu\nu}\nabla_\mu\nabla_\nu$ and $R(g)$ is the scalar curvature. The cubic term has the defocusing sign and its coefficient has been normalised by rescaling the field. The curvature coupling and the cubic nonlinearity make \eqref{eq:ConformalCovariant} conformally invariant in four spacetime dimensions.

The background metric of the Einstein cylinder is
\begin{equation}
	\label{eq:EinsteinCylinderMetric}
	g=-\mathrm{d}t^{2}+\mathrm{d}x^{2}+\sin^{2}x\left(\mathrm{d}\theta^{2} +\sin^{2}\theta\,\mathrm{d}\varphi^{2}\right)\,,
\end{equation}
where $t\in\mathbb{R}$, $x\in[0,\pi]$, $\theta\in[0,\pi]$, and $\varphi\in[0,2\pi)$. The radius of the spatial three-sphere is normalised to one, so that $R(g)=6$.

We restrict attention to spherically symmetric fields, $\phi(t,x,\theta,\varphi)=w(t,x)$. Equation \eqref{eq:ConformalCovariant} then reduces to
\begin{equation}
	\label{eq:ConformalPhysicalTime}
	\partial_t^{2}w-\frac{1}{\sin^{2}{x}}\partial_x\left(\sin^{2}{x}\,\partial_x w\right)+w+w^{3}=0\,,
	\qquad x\in(0,\pi)\,.
\end{equation}
Smoothness of $\phi$ at the poles requires $w$ to admit smooth even extensions across $x=0$ and $x=\pi$. In particular,
\begin{equation}
	\partial_x w(t,0)=\partial_x w(t,\pi)=0\,.
\end{equation}
Thus, the endpoint conditions express regularity on $\mathbb{S}^{3}$.

Conformal invariance relates this model to a self-interacting scalar field on four-\linebreak dimensional anti-de Sitter spacetime. This correspondence connects the model to questions of nonlinear dynamics and stability in general relativity with a negative cosmological constant; for details, see e.g. \cite{Bizoń.2014}.

For regular solutions, the conserved energy is
\begin{equation}
	\label{eq:ConformalEnergy}
    E[w]=\int_0^\pi
    \left[
        \frac{1}{2}(\partial_t w)^2
        +\frac{1}{2}(\partial_x w)^2
        +\frac{1}{2}w^2
        +\frac{1}{4}w^4
	\right]\sin^2x\,\mathrm{d}x\,.
\end{equation}
The equation also admits a breathing mode whose modulus is conserved \cite{Evnin.2020fys}. For the solutions considered here, which have definite parity under reflection across the equator, this quantity vanishes identically; see \cite{Ficek.2025a}.

An equivalent formulation is obtained by setting
$u(t,x)=\sin x\,w(t,x)$:
\begin{equation}
	\label{eq:ConformalDirichlet}
	\partial_t^2u-\partial_x^2u+\frac{u^3}{\sin^2x}=0,
	\qquad
	u(t,0)=u(t,\pi)=0\,.
\end{equation}
This formulation is used, for example, in \cite{Bizoń.2017ndf, Ficek.2025a}. The Dirichlet conditions follow from the regularity of the original field; more precisely, a smooth spherically symmetric field gives rise to smooth odd extensions of $u$ across the endpoints. We work with $w$ because, although this formulation introduces geometric coefficients into the linear operator, it retains the simple nonlinearity $w^3$. This form is convenient for the functional setting and operator estimates developed below.

Our aim is to establish the existence of nontrivial, regular, finite-energy time-periodic solutions of \eqref{eq:ConformalPhysicalTime}. Writing $\tau=\Omega t$, with $\Omega>0$, we seek solutions of
\begin{equation}
	\label{eq:conformal}
	Lw+w^3=0\,,
	\qquad
	w(\tau+2\pi,x)=w(\tau,x)\,,
\end{equation}
where
\begin{equation}
	\label{eq:26.10.08_02}
	L=\Omega^2\partial_\tau^2-\partial_x^2-2\cot x\,\partial_x+1\,.
\end{equation}

For a specified value of $\Omega$, we prove the existence of three distinct solutions by validating three explicit approximations. The proof combines analytical operator estimates with computations in exact rational arithmetic that verify a contraction criterion and establish the pairwise distinctness of the resulting solutions.

We note that relaxing spherical symmetry introduces angular dependence and leads to solutions with a much richer structure \cite{Maliborski.2016, Rostworowski.2017, Bizon.Personal2018}. In particular, the axially symmetric problem requires additional analysis, which will be presented in a separate work \cite{Ficek.Future}.

\subsection{Main result}
\label{sec:MainResult}

Our main result is summarised in the following theorem.

\begin{theorem}\label{thm:main}
    For $\Omega=23/8$ and each of $v^{(i)}$, $i=1,2,3$ given explicitly in the Supplemental Material, there exists a corresponding solution $w^{(i)}$ to Eq.~\eqref{eq:conformal} close to $v^{(i)}$ in the sense $\Vert w^{(i)}-v^{(i)} \Vert<\varepsilon^{(i)}$, where $\Vert \cdot \Vert$ is a weighted $\ell^{1}$ norm defined in Section~\ref{sec:FunctionalSpaces}, and $\varepsilon$'s are listed in \eqref{eq:26.10.08_01}. The solutions $\{\pm w^{(i)}\}_{i=1}^3$ are pairwise distinct, i.e., $w^{(i)}\ne\pm w^{(j)}$ for $i\ne j$.
\end{theorem}

This is the first result of this kind for a Hamiltonian PDEs in higher dimensions. In fact, this is the first proof of existence of large periodic solutions of the problem directly related to the analogous gravitational system.

The profiles of approximate solutions $v^{(i)}$ are shown in Figure~\ref{fig:3dplots}. Note that even though the plots for $v^{(2)}$ and $v^{(3)}$ are very similar, they correspond to distinct solutions, as stated in Theorem~\ref{thm:main}.

Let us emphasise that our approach is very robust and can be directly adapted to prove the existence of other solutions to this problem, as long as their frequency has a form $\Omega=\frac{2p+1}{q}$, where $p\in\mathbb{N}$, $q\in\mathbb{N}_+$, see Section~\ref{sec:preliminaries} for the details.

\begin{figure}[tbp]
  \centering
   \includegraphics[width=\textwidth]{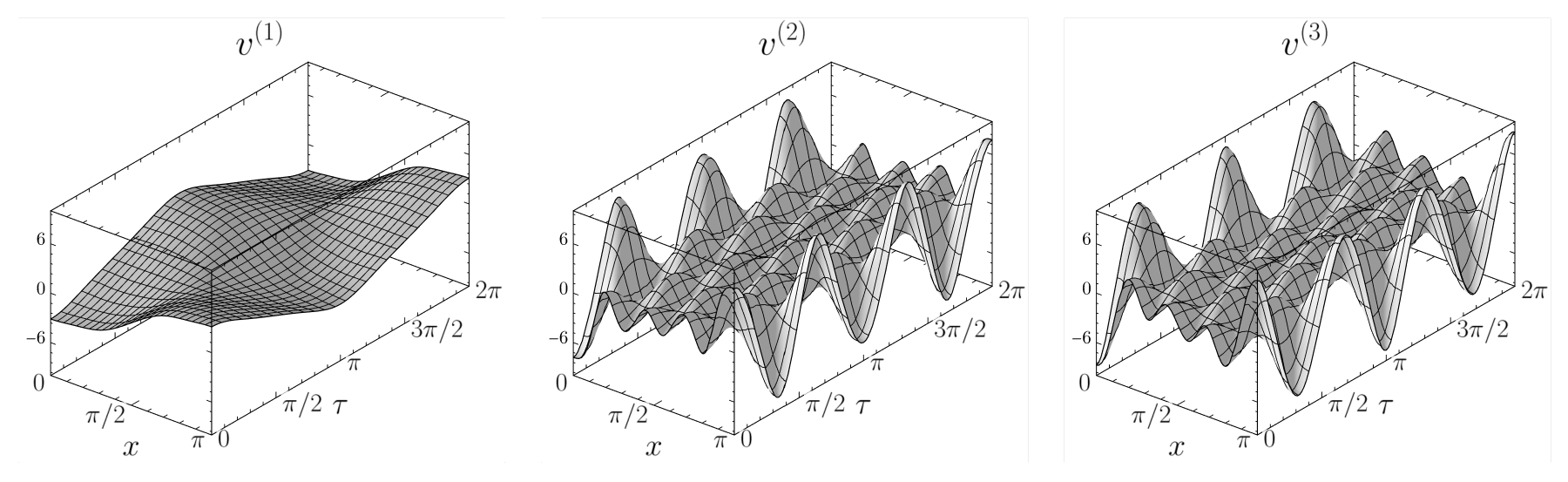}
  \caption{Three-dimensional plots of the approximate time-periodic solutions $v^{(i)}$, $i=1,2,3$, of \eqref{eq:conformal}-\eqref{eq:26.10.08_02}, with common frequency $\Omega=23/8$. The truncation sizes are $(M,N)=(10,20)$, $(42,84)$, and $(38,76)$ for the first, second, and third solutions, respectively. The corresponding approximate energies~\eqref{eq:ConformalEnergy} are $18.85$, $95.73$, and $97.44$.}
  \label{fig:3dplots}
\end{figure}

\subsection{Previous work}

The interest in time-periodic solutions of Hamiltonian PDEs has a rather long history \cite{berti2007nonlinear}. Classical proofs of existence were limited to simple models, like one-dimensional cubic wave equation on an interval, and due to the intrinsic complexity and technical obstacles were necessarily restricted to small amplitude solutions with existence stated on a Cantor sets only. A series of works analysed small periodic solutions to the conformal equation \cite{Chatzikaleas.2020oh, Berti.2024, Silimbani.2024, Chatzikaleas.2024}, its variants with other nonlinearities \cite{Berti.2024}, or different than spherical symmetry reductions (Hopf fibration) \cite{Berti.2024, Chatzikaleas.2024}.

Nevertheless, our recent numerical results on the wave equation and the conformal equation suggest that periodic solutions exist for arbitrary amplitudes. Moreover, those solutions form extraordinary patterns on energy-frequency diagram, which are consequences of distinct Fourier modes dominating others due to nonlinear resonances. Those observations, were analysed by the authors in a series of works: \cite{Ficek.2024A, Ficek.2024B} and \cite{Ficek.2025a, Ficek.2026proc} for cubic wave equation and the conformal equation studied here, respectively, and also the cubic beam equation \cite{ficek2026nonlinear}.
Moreover, to put the numerical results on a firmer footing in \cite{ficek2025new} we proved existence of a specific large-amplitude solutions of the cubic wave equation by adapting the idea of \cite{Arioli.2017}. This proof uses computer assistance to verify certain operator bounds for a set of some concrete approximate solutions, ultimately proving the existence of true solutions in their neighbourhood. 
In addition, the first author used a similar approach to get an elementary proof of existence of perturbative solutions, without computer assistance \cite{Ficek2025timeperiodic}.

\subsection{Novel aspects}

In the present work we extend the analysis and techniques developed in \cite{ficek2025new} to the conformal equation \eqref{eq:ConformalPhysicalTime} and prove the existence of large solutions first reported in \cite{Ficek.2025a}. This requires substantial new insight and effort as a direct application of the previous approach is not feasible.
Compared to the cubic wave equation, using eigenbasis of the linearised problem leads to the issues with control of the nonlinearity.
Instead, we use simple basis functions introduced below. The price to pay is that the linear operator \eqref{eq:26.10.08_02} is no longer diagonal, so the control on its inverse requires substantially more work.
Thus, contrary to the cubic wave equation, significant effort was devoted into understanding the $L^{-1}$ to optimise the number of sites where direct computations are needed, consequently extending the applicability of analytical bounds.
Moreover, another crucial improvement is a reduction of operators entering the analysis, the $\Lambda$ and $A$ introduced in Section~\ref{sec:FixedPointFormulation}, allowing us to speed up, and eventually, make the computations possible in a reasonable time on moderate computers. Thus, this work presents a substantial advancement compared to \cite{ficek2025new} and serves as a crucial step toward addressing even more complex problems in the near future.

\subsection{Outline}
\label{sec:Outline}

Our line of work is similar to the one employed in \cite{ficek2025new}. We reformulate the equation as a specially tailored fixed point problem in an appropriate functional space. Then, we show that in sufficiently small neighbourhoods of the approximate solutions (found numerically) relevant operators are contractions, letting us apply the Banach contraction principle. This strategy is described in greater detail in Section~\ref{sec:framework}. The main difficulty is posed by estimating the norm of the linear operator $L^{-1}$. By deriving proper analytical bounds in Section \ref{sec:estimates}, we are able to reduce it to executing a finite number of calculations that are performed by a computer, as described in Section~\ref{sec:ComputeAssistedVerification}.

\section{Framework and strategy of the proof}
\label{sec:framework}

\subsection{Preliminaries}
\label{sec:preliminaries}

Instead of working in the eigenbasis of a linear operator \eqref{eq:26.10.08_02}, we introduce simple trigonometric basis functions
\begin{align}
    P_{m,n}=\cos(2m+1)\tau\, \cos(2n+1)x\,,
\end{align}
with $m,n\in\mathbb{N}$. Let us also define the function
\begin{align}
    \alpha(m,n)=-\Omega^2(2m+1)^2+4(n+1)^2\,
\end{align}
for natural numbers $m$ and $n$. Then, a short calculation leads to the following decomposition
\begin{align}
    \cot x\, \partial_x P_{m,n}=-(2n+1)\left(P_{m,n}+2\sum_{k=0}^{n-1}P_{m,k}\right)\,.
\end{align}
It allows us to write the action of the linear operator $L$ on the basis functions as
\begin{align}
    L P_{m,n}=\alpha(m,n)P_{m,n} +4(2n+1)\sum_{k=0}^{n-1} P_{m,k}\,.
\end{align}
In opposition to the cubic wave equation studied in \cite{ficek2025new}, $L$ is no longer a diagonal operator in the basis $P_{m,n}$, as foreshadowed in Section~\ref{sec:Introduction}. Nevertheless, since it has a triangular structure, its inverse can be written formally as
\begin{align}\label{eqn:Linv}
    L^{-1} P_{m,n}=\frac{1}{\alpha(m,n)}P_{m,n} -\frac{4(2n+1)}{\alpha(m,n-1)\alpha(m,n)}\sum_{k=0}^{n-1} P_{m,k}\,.
\end{align}
The validity of this expression depends on the properties of function $\alpha$ and is connected with the sound problem of small divisors \cite{arnold1988dynamical}. We overcome it, by restricting ourselves to frequencies having the form of $\Omega=\frac{2p+1}{q}$, where $p\in\mathbb{N}$, $q\in\mathbb{N}_+$. Then the function $\alpha$ becomes
\begin{align}\label{eqn:alpha_pq}
    \alpha(m,n)=\frac{1}{q^2}\left[-(2p+1)^2(2m+1)^2+4q^2(n+1)^2\right].
\end{align}
Since it involves a difference of an odd and even number, we have $|\alpha(m,n)|\geq q^{-2}>0$. As a result, $L^{-1}$ is well defined. In Section \ref{sec:boundsL-1} we also show its compactness.

\subsection{Functional setting}
\label{sec:FunctionalSpaces}

We define the functional space $X$ in an analogous way to the construction presented in \cite{ficek2025new} (see also \cite{Arioli.2017}). Let us introduce an auxiliary space $\tilde{X}$ that is spanned by finite linear combinations of the basis vectors $P_{m,n}$. Its elements $w\in\tilde{X}$ can be uniquely represented as finite sums $w=\sum_{m,n}\hat{w}_{m,n}P_{m,n}$. Then we define a norm on $\tilde{X}$ via
\begin{align}\label{eqn:norm}
    \Vert w \Vert = \sum_{m,n} \rho_\tau^{2m+1}\rho_x^{2n+1}|\hat{w}_{m,n}|\,.
\end{align}
where $\rho_{\tau},\rho_x>1$. Our target functional space $X$ is a completion of $\tilde{X}$ with respect to this norm. Let us note that all elements of $X$ are differentiable in real variables thanks to the weights. It also allows us to control the nonlinearity since for any $u,v,w\in X$, their pointwise product also belongs to $X$ with $\Vert uvw\Vert\leq \Vert u\Vert \Vert v\Vert \Vert w\Vert$. Finally, norms of linear operators $\mathcal{H}:X\to X$ can be estimated in this functional space quite easily since
\begin{align}\label{eqn:normOp}
\Vert\mathcal{H}\Vert=\sup_{(m,n)\in\mathbb{N}^2}\frac{\Vert \mathcal{H}(P_{m,n})\Vert}{\Vert P_{m,n}\Vert}=\sup_{(m,n)\in\mathbb{N}^2}\frac{\Vert \mathcal{H}(P_{m,n})\Vert}{\rho_\tau^{2m+1}\rho_x^{2n+1}}\,.
\end{align}
For proofs of these facts, we refer to \cite{ficek2025new}.

For future convenience, let us define $P_{m,n}$ with negative indices as
\begin{equation}\label{eqn:negative_P}
        P_{-m,n}=P_{m-1,n}\,,\qquad
        P_{m,-n}=P_{m,n-1}\,,\qquad
        P_{-m,-n}=P_{m-1,n-1}\,,
\end{equation}
where $(m,n)\in\mathbb{N}^2$. In particular, it means that for any $(m,n)\in\mathbb{Z}^2$, $\Vert P_{m,n}\Vert\leq\Vert P_{|m|,|n|}\Vert$.

\subsection{Proof of the main result}
\label{sec:FixedPointFormulation}

Let us assume that $v\in X$ is an approximate solution to Eq.~\eqref{eq:conformal} with frequency $\Omega=\frac{2p+1}{q}$ (where $p\in\mathbb{N}, q\in\mathbb{N}_+$). We define the following nonlinear function $\mathcal{N}:X\to X$
\begin{align}
    \mathcal{N}(h):=-L^{-1}(v+Ah)^3-v+(I-A)h\,,
\end{align}
where $A$ is a certain bounded linear operator. It is easy to check that any fixed point $h$ of $\mathcal{N}$ provides us with a solution $w=v+Ah$ to Eq.~\eqref{eq:conformal}. Our goal is to construct an operator $A$ such that $\mathcal{N}$ is a contraction inside some small ball $B_{\delta}(0)$. Let us represent the multiplication of elements in $X$ by $v^2$ as an operator $\Lambda:X\to X$, so
\begin{equation}
    \Lambda h=v^2h\,.
\end{equation}
Then, we can write 
\begin{align}
    \mathcal{N}(h):=\mathcal{N}(0)+H(h)+\mathcal{O}(h^2)\,,
\end{align}
where 
\begin{align}\label{eqn:H}
    H:=I-\left(I+3L^{-1}\circ \Lambda\right)\circ A\,.
\end{align}
A necessary condition for $\mathcal{N}$ to be a contraction in a neighbourhood of zero, is for $H$ to satisfy $\Vert H\Vert<1$, i.e., $\Vert H(P_{m,n})\Vert/\Vert P_{m,n}\Vert$ are strongly below one for every $(m,n)\in\mathbb{N}^2$. It would suggest $A(h)=(I+3L^{-1}\circ \Lambda)^{-1}$ as a good choice. However, finding such an inverse explicitly proves to be a very challenging problem. Instead, we construct operators $A$ in such a way that they approximate such inverses. For the future convenience by $H_0$ we will denote the operator $H$ with $A=I$, so
\begin{align}
    H_0=-3L^{-1}\circ\Lambda\,.
\end{align}

Let us now assume that the following technical result holds.
\begin{theorem}\label{thm:cap}
    For $\Omega=23/8$, $\rho_\tau=1+10^{-20}$, $\rho_x=11/10$, and each of $v\in\{v^{(1)},v^{(2)},v^{(3)}\}$, provided explicitly in the Supplemental Material, there exists a respective linear isomorphism $A:X\to X$ and resulting $H$, given by Eq.~\eqref{eqn:H}, such that one can find positive constants $\delta$ and $\Kzero$ satisfying
\begin{equation}
    \label{eq:25.05.21_05}
    \Vert H\Vert+6\left\Vert L^{-1}\right\Vert \Vert v\Vert \Vert A \Vert^2 \delta+3\left\Vert L^{-1}\right\Vert \Vert A \Vert^3 \delta^2 <\Kzero<1\,,\qquad
    \Vert \mathcal{N}(0)\Vert <(1-\Kzero)\delta
    \,.
\end{equation}
In addition, if we introduce $\varepsilon:=\Vert A\Vert\delta$, then the values $\varepsilon^{(i)}$ for respective $v^{(i)}$ satisfy
\begin{align}\label{eq:u-u0}
   \left\Vert v^{(i)} \pm v^{(j)}\right\Vert >\varepsilon^{(i)}+\varepsilon^{(j)}\,,
\end{align}
for $i,j\in\{1,2,3\}$ and $i\neq j$.
\end{theorem}

Then the Theorem~\ref{thm:main} immediately follows from Theorem~\ref{thm:cap}.

\begin{proof}[Proof of Theorem \ref{thm:main}]
Let us note that for any $h_1, h_2\in X$ we have
\begin{align}
    \mathcal{N}(h_1)-\mathcal{N}(h_2)=&H(h_1-h_2)-3L\left(v\,(Ah_1+Ah_2)(Ah_1-A h_2)\right)\\
    &-L\left[\left((Ah_1)^2+Ah_1\,Ah_2+(Ah_2)^2\right)(Ah_1-A h_2)\right]\,.
\end{align}
When restricting to $h_1,h_2\in B_\delta(0)$, it gives us
\begin{align}\label{eq:NN_norm}
    \left\Vert\mathcal{N}(h_1)-\mathcal{N}(h_2)\right\Vert\leq
    \left(\Vert H\Vert+6\left\Vert L^{-1}\right\Vert \Vert v\Vert \Vert A \Vert^2 \delta+3\left\Vert L^{-1}\right\Vert \Vert A \Vert^3 \delta^2 \right)\Vert h_1-h_2\Vert\,.
\end{align}
When $A$ is chosen according to the Theorem \ref{thm:cap}, we have
\begin{align*}
    \left\Vert\mathcal{N}(h_1)-\mathcal{N}(h_2)\right\Vert<
    \Kzero\Vert h_1-h_2\Vert\,,
\end{align*}
as well as
\begin{align}
\label{eq:25.05.21_04}
    \left\Vert\mathcal{N}(h)\right\Vert\leq\left\Vert\mathcal{N}(h)-\mathcal{N}(0)\right\Vert+ \left\Vert\mathcal{N}(0)\right\Vert < \Kzero \delta+(1-\Kzero)\delta=\delta\,.
\end{align}
As a result, the restriction $\mathcal{N}|_{B_\delta(0)}:B_{\delta}(0)\to B_\delta (0)$ is a contraction and we can employ the Banach contraction principle to find the fixed point $h\in B_\delta (0)$ of $\mathcal{N}$. It leads to the solution $w$ of \eqref{eq:conformal} expressed as $w=v+Ah$. Since $h\in B_\delta(0)$, the distance between the solution $w$ and its approximation $v$ is bounded by
\begin{align}\label{eq:epsilon}
   \Vert w-v\Vert\leq\Vert A h\Vert\leq \Vert A \Vert \delta = \varepsilon\,.
\end{align}
Then the fact that the solutions $\{\pm w^{(i)}\}_{i=1}^3$ are distinct follows from \eqref{eq:u-u0}, since for $i,j\in\{1,2,3\}$ and $i\neq j$ we have
\begin{align}
   \left\Vert w^{(i)}\pm w^{(j)}\right\Vert\geq \left\Vert v^{(i)}\pm v^{(j)}\right\Vert-\left\Vert w^{(i)}-v^{(i)}\right\Vert-\left\Vert w^{(j)}-v^{(j)}\right\Vert>0\,.
\end{align}
\end{proof}

The remaining part of this paper is devoted to the proof of Theorem \ref{thm:cap}. The main goal consists of constructing $A$ such that $\Vert H\Vert<\sigma$, where $\sigma$ is some fixed number belonging to the interval $(0,1)$. As Eq.~\eqref{eqn:normOp} shows, it can be done by establishing a common bound working for all basis elements $P_{m,n}$. This effort can be divided into two parts. First, we use analytical estimates to show that outside of a finite family of $P_{m,n}$, that we will call $\mathcal{S}_{H_0}$, it holds $\Vert H_0 P_{m,n}\Vert/\Vert P_{m,n}\Vert<\sigma$. As a result, for almost all basis elements it is sufficient to have $A=I$. We are left with a finite number of $P_{m,n}$ belonging to $\mathcal{S}_{H_0}$ that require further work. We use a computer to calculate the value of $\Vert H_0 P_{m,n}\Vert/\Vert P_{m,n}\Vert$ for each of them. If this value exceeds $\sigma$, we add such $P_{m,n}$ to the set $\mathcal{S}_A$. Let $\Pi_A$ denote a projection on a linear subspace spanned by elements of $\mathcal{S}_A$. We construct a suitable operator $\mathcal{A}$ for this subspace, so that with 
\begin{align}\label{eqn:operatorA}
    A=\mathcal{A}\circ \Pi_{A}+(I-\Pi_A)
\end{align}
we are able to achieve $\Vert H P_{m,n}\Vert/\Vert P_{m,n}\Vert<\sigma$ at every $P_{m,n}$.

\section{Operator estimates}
\label{sec:estimates}

\subsection{Bounds for \texorpdfstring{$L^{-1}$}{L-1}}
\label{sec:boundsL-1}

The norm of any basis vector is given by $\Vert P_{m,n} \Vert=\rho_\tau^{2m+1}\rho_x^{2m+1}$. Thus, the explicit expression \eqref{eqn:Linv} lets us write
\begin{align*}
    \left\Vert L^{-1}P_{m,n}\right\Vert &= \left|\frac{1}{\alpha(m,n)}\right|\rho_\tau^{2m+1}\rho_x^{2n+1}+\left|\frac{4(2n+1)}{\alpha(m,n-1)\alpha(m,n)}\right|\sum_{k=0}^{n-1}\rho_\tau^{2m+1}\rho_x^{2k+1}\nonumber\\
    &=\left|\frac{1}{\alpha(m,n)}\right|\rho_\tau^{2m+1}\rho_x^{2n+1}+\left|\frac{4(2n+1)}{\alpha(m,n-1)\alpha(m,n)}\right|\frac{\rho_\tau^{2m+1}\rho_x(\rho_x^{2n}-1)}{\rho_x^2-1}\,.
\end{align*}
As a result, expressions required for the computation of $\Vert L^{-1}\Vert$ via Eq.~\eqref{eqn:normOp} can be written as
\begin{align}\label{eqn:normLinv}
    \frac{\left\Vert L^{-1}P_{m,n}\right\Vert}{\Vert P_{m,n} \Vert} &= \left|\frac{1}{\alpha(m,n)}\right|+\left|\frac{4(2n+1)}{\alpha(m,n-1)\alpha(m,n)}\right|\frac{1-\rho_x^{-2n}}{\rho_x^2-1}\,.
\end{align}
Most of this subsection is devoted to finding bounds on this expression in various regimes.

First, recall the following lemma (its proof can be found in \cite{ficek2025new})
\begin{lemma}\label{lem:ineq}
For $|t-s|\geq 1$ it holds
\begin{align}
    |t^2-s^2|\geq 2\max\{|t|,|s|\}-1\,.
\end{align}
\end{lemma}

This result lets us bound $|\alpha(m,n)|$ from below, since the expression inside the square bracket in \eqref{eqn:alpha_pq} is a difference of two perfect squares of difference parities. As a result, we have
\begin{align}
    |\alpha(m,n)|&\geq\frac{1}{q^2}\left(2\max\{(2p+1)(2m+1),2q(n+1)\}-1\right)\,.
\end{align}
In addition, for $n\geq1$ it gives us a simple estimate for $4(2n+1)/|\alpha(m,n-1)|\leq 12/(4q-1)$. Together, these bounds yield
\begin{align}\label{eqn:normL-1}
        \frac{\left\Vert L^{-1}P_{m,n}\right\Vert}{\Vert P_{m,n} \Vert} \leq \frac{q^2}{2\max\{(2p+1)(2m+1),2q(n+1)\}-1}\left(1+\frac{12q^2}{(4q-1)(\rho_x^2-1)}\right)\,,
    \end{align}
It leads to the following conclusion.
\begin{obs}
    The operator $L^{-1}$ is compact.
\end{obs}

As the next step, we will show that $\Vert L^{-1}P_{m,n}\Vert/\Vert P_{m,n}\Vert$ rapidly decreases as $(m,n)$ drives further from the line $n=\frac{2p+1}{q}m$.

\begin{lemma}\label{lem:band1}
Let us fix any $m\in \mathbb{N}$. Then, for any $n\in\mathbb{N}$ such that $n-\frac{2p+1}{q}m=\eta>\frac{2p+1}{2q}$, it holds
\begin{align}\label{eqn:bound_eta}
   \frac{\left\Vert L^{-1}P_{m,n}\right\Vert}{\Vert P_{m,n} \Vert} &\leq \frac{4 \eta  (3 \eta +1)-\left(\frac{2 p+1}{q}\right)^2}{16 \left(\eta -\frac{2 p+1}{2 q}\right) \left(\eta +1-\frac{2 p+1}{2 q}\right) \left(\eta +\frac{2 p+1}{2 q}\right) \left(\eta+1 + \frac{2 p+1}{2 q}\right)}\,.
\end{align}
\end{lemma}

\begin{proof}
Since
\begin{align*}
    \frac{1-\rho_x^{-2n}}{\rho_x^2-1}=\frac{1}{\rho_x+1}\sum_{j=1}^{2n} \frac{1}{\rho_x^j}\leq\frac{2n}{\rho_x+1}\leq  n\,,
\end{align*}
from Eq.~\eqref{eqn:normLinv} we have
\begin{align}\label{eqn:Linv_bound}
    \frac{\left\Vert L^{-1}P_{m,n}\right\Vert}{\Vert P_{m,n} \Vert} &\leq \left|\frac{1}{\alpha(m,n)}\right|+\left|\frac{4(2n+1)n}{\alpha(m,n-1)\alpha(m,n)}\right|\,.
\end{align}
Now, let us assume that $n=\frac{2p+1}{q}m+\eta$ with $\eta>\frac{2p+1}{2q}$. Then
\begin{multline*}
    \alpha(m,n-1)=4\left(\frac{2p+1}{q}m+\eta\right)^2-\frac{(2p+1)^2}{q^2}(2m+1)^2
    \\
    =4\left(\eta +\frac{(2p+1)}{2q}(4 m+1) \right) \left(\eta-\frac{2p+1}{2q}\right)\,.
\end{multline*}
Since $\eta> \frac{2p+1}{2q}$, $\alpha(m,n-1)>0$, and additionally also $ \alpha(m,n)>0$. Thus, 
\begin{multline}
    \frac{\left\Vert L^{-1}P_{m,n}\right\Vert}{\Vert P_{m,n} \Vert}  \leq \frac{1}{\alpha(m,n)}+\frac{4(2n+1)n}{\alpha(m,n)\alpha(m,n-1)}\\
    =\frac{\left(8 m^2-4 m-1\right) \left(\frac{2 p+1}{q}\right)^2+4 (6 \eta +1) m\frac{2 p+1}{q}+ 4 \eta  (3 \eta +1)}{16 \left(\eta -\frac{2 p+1}{2 q}\right) \left(\eta +1-\frac{2 p+1}{2 q}\right) \left(\eta +(4 m+1) \frac{2 p+1}{2 q}\right) \left(\eta+1 +(4 m+1) \frac{2 p+1}{2 q}\right)}\,.
\end{multline}
This expression is decreasing in $m$ (under fixed $\eta$). Additionally, since $\eta>\frac{2p+1}{2q}$, for fixed $m$ it is decreasing in $\eta$. It lets us write an explicit bound \eqref{eqn:bound_eta}.
\end{proof}

\begin{lemma}\label{lem:band2}
Let us fix any $n\in\mathbb{N}$, then for any $m\in\mathbb{N}$ such that $m-\frac{q}{2p+1}n=\chi>\max\{\frac{q}{2p+1}-\frac{1}{2},0\}$ it holds
\begin{align}\label{eqn:bound_chi}
   \frac{\left\Vert L^{-1}P_{m,n}\right\Vert}{\Vert P_{m,n} \Vert} &\leq \frac{1}{\left(\frac{2 p+1}{q}(2 \chi +1)-2\right) \left(\frac{2 p+1}{q}(2 \chi +1)+2\right)}\,.
\end{align}
\end{lemma}

\begin{proof}
Let $m=\frac{q}{2p+1}n+\chi$, where $\chi> \max\{\frac{q}{2p+1}-\frac{1}{2},0\}$, then
\begin{align*}
    \alpha(m,n)&=-\frac{2(2p+1)}{q}\left(2(2n+1)+\frac{2p+1}{q}(2\chi+1)\right)\left[\chi-\left(\frac{q}{2p+1}-\frac{1}{2}\right)\right]\,.
\end{align*}
Since $\chi> \frac{q}{2p+1}-\frac{1}{2} $, this expression is negative. The same conclusion holds for $ \alpha(m,n-1)$, so using \eqref{eqn:Linv_bound} we get
\begin{multline*}
    \frac{\left\Vert L^{-1}P_{m,n}\right\Vert}{\Vert P_{m,n} \Vert}  \leq -\frac{1}{\alpha(m,n)}+\frac{4(2n+1)n}{\alpha(m,n)\alpha(m,n-1)}\\
    =\frac{\left(\frac{2 p+1}{q}(2 \chi +1)+2n\right)^2+4 (n+1) n}{\frac{2p+1}{q}(2\chi+1)\left(\frac{2p+1}{q}(2\chi+1)-2\right)\left(\frac{2p+1}{q}(2\chi+1)+4n\right)\left(\frac{2p+1}{q}(2\chi+1)+2(2n+1)\right)}\,.
\end{multline*}
For fixed $\chi$, it is decreasing in $n$. Since $\chi>\frac{q}{2p+1}-\frac{1}{2}$, it also decreases in $\chi$ for fixed $n$. Hence, we can bound it at $n=0$, getting \eqref{eqn:bound_chi}.
\end{proof}

Since $\Omega$ has been chosen to be a rational number of the form $\frac{2p+1}{q}$, the line $n=\Omega m$ admits a certain periodic structure: if it goes across a point $(m,n)$ belonging to the $\mathbb{N}^2$ lattice, it also crosses $(m+qk,n+(2p+1)k)$ for any $k\in\mathbb{Z}$. The final result of this section concerns behaviour of $\Vert L^{-1}P_{m,n}\Vert/\Vert P_{m,n}\Vert$ as we move along the lattice by such fixed intervals. First, let us introduce the following useful notion. Consider the triple $(p,q,\rho_x)$ with $p\in\mathbb{N}$, $q\in\mathbb{N}_+$, and $\rho_x>1$, and define
\begin{align*}
    \Delta=\frac{(2p+1)}{q}(2m+1)-2n,\qquad\nu_k=n+(2p+1)k,\qquad S_k=\sum_{j=1}^{\nu_k} \rho_x^{-2j}\,.
\end{align*}
We  say that the triple $(p,q,\rho_x)$ satisfies condition $(\ast)$ when, for every $(m,n)\in\mathbb{N}^2$ such that $\Delta>0$, the expression
\begin{multline}
    \label{eqn:condition_f}
    \Delta  (\Delta +4 \nu_k) (\Delta +4 \nu_{k+1}) (\nu_{k+1}-\nu_k)
    +(2 \nu_k+1) S_k (\Delta +4 \nu_{k+1}) (\Delta +4 \nu_{k+1}+2)
    \\
    -(2 \nu_{k+1}+1) S_{k+1} (\Delta +4 \nu_{k}) (\Delta +4 \nu_{k}+2)\,,
\end{multline}
is positive for every $k\in\mathbb{N}$.

\begin{lemma}\label{lem:f_decreasing}
    Let $(p,q,\rho_x)$ satisfy condition $(\ast)$. Then for any fixed $(m,n)\in\mathbb{N}^2$, the function
    \begin{align}\label{eqn:function_f}
        f_{m,n}(k)=\frac{\left\Vert L^{-1} P_{m+qk,n+(2p+1)k}\right\Vert}{\left\Vert P_{m+qk,n+(2p+1)k}\right\Vert}\,.
    \end{align}
    is decreasing in $k\in\mathbb{N}$.
\end{lemma}
\begin{proof}
Using Eq.~\eqref{eqn:normLinv}, we can write explicitly (note that $\Delta+2\nu_k$ is positive)
\begin{align*}
    f_{m,n}(k)=\frac{1}{|\Delta-2|}\left[\frac{1}{4\nu_k+\Delta+2}+\frac{4(2\nu_k+1)}{|\Delta|(4\nu_k+\Delta)(4\nu_k+\Delta+2)}S_k\right]\,.
\end{align*}
Hence, it is sufficient to consider only the monotonicity of the expression inside the square brackets. 

If $\Delta<0$, the first term inside the bracket is clearly decreasing. The second term can be rearranged to
\begin{align*}
    \frac{4(2\nu_k+1)\nu_k}{|\Delta|(4\nu_k+\Delta)(4\nu_k+\Delta+2)} \frac{S_k}{\nu_k}\,.
\end{align*}
The expression $S_k/\nu_k$ is a mean of the elements included in the sum $S_k$. Since $\rho_x^{-2j}$ decreases with $j$, $S_k/\nu_k$ is decreasing in $k$. To study the monotonicity of the first factor, we simply differentiate it with respect to $\nu_k$, resulting in
\begin{align*}
    \frac{4 \Delta \left((4 \nu_k+1) (\Delta+2 \nu_k)+8 \nu_k^2+6 \nu_k+2\right)}{(4\nu_k+\Delta)^2 (4\nu_k+\Delta+2)^2}\,.
\end{align*}
It is negative since $\Delta<0$. Hence, $f_{m,n}(k)$ is a decreasing function.

Now assume that $\Delta>0$ and consider $f_{m,n}(k)-f_{m,n}(k+1)$. By putting this difference over a common denominator, we get a product of positive expressions and the expression \eqref{eqn:condition_f}. Since condition $(\ast)$ holds, it is positive, so $f_{m,n}(k)$ is decreasing. 
\end{proof}

Lemmas \ref{lem:band1}, \ref{lem:band2}, and \ref{lem:f_decreasing} lead to the following picture. In the quadrant $(m,n) \in \mathbb{N}^2$, there is a band along the line $n=\frac{2p+1}{q}m$ such that all points $(m,n)$ with large $\Vert L^{-1}P_{m,n}\Vert/\Vert P_{m,n}\Vert$ necessarily lie in this band. As we move further away from this band, \eqref{eqn:normLinv} decreases quadratically. There is also an additional lattice structure: if the condition $(\ast)$ is satisfied, then when moving from any given point $(m,n)$ to $(m+q,n+2p+1)$, the expression \eqref{eqn:normLinv} decreases. We will use these observations to derive bounds on $H_0$ in Section~\ref{sec:BoundsOnH0}.

\subsection{Bounds for \texorpdfstring{$\Lambda$}{Lambda}}
\label{sec:lambda}
In this short subsection we outline how the operator $\Lambda$ acts on basis vectors. Let us assume that $v\in X$ has a decomposition 
\begin{equation}\label{eq:25.06.03_01}
    v=\sum_{\substack{0\leq m< M\\ 0\leq n<N}}\hat{v}_{m,n}P_{m,n}\,.
\end{equation}
Let us introduce
\begin{equation}\label{eq:25.06.03_02}
    \hat{M}= 2M-1 \,,\qquad \hat{N}= 2N-1 \,.
\end{equation}
Then, it is easy to see that since $P_{m,n}$ are particular trigonometric functions, it simply holds
$\Lambda P_{m,n}\subset \textup{span}\,\mathcal{R}_{\hat{M},\hat{N}}^{(m,n)}$, where
\begin{align}\label{eqn:Rcal}
    \mathcal{R}_{\hat{M},\hat{N}}^{(m,n)}:= \{P_{\hat{m},\hat{n}}:\ \max\{0,m-\hat{M}\}\leq \hat{m}\leq m+\hat{M}\ \land\ \max\{0,n-\hat{N}\}\leq \hat{n}\leq n+\hat{N}\}\,.
\end{align}
In addition, we have the following decomposition:
\begin{align}\label{eq:uvP_decompositionMN}
    \Lambda P_{\hat{M},\hat{N}}= \sum_{\hat{m}=0}^{2\hat{M}}\sum_{\hat{n}=0}^{2\hat{N}}c_{\hat{m},\hat{n}}P_{\hat{m},\hat{n}}\,,
\end{align}
where the coefficients $c_{\hat{m},\hat{n}}$ depend only on $v$. For any other pair of natural numbers $(m,n)$, we can write
\begin{align}\label{eq:uvP_decomposition}
    \Lambda P_{m,n} = \sum_{\hat{m}=0}^{2\hat{M}}\sum_{\hat{n}=0}^{2\hat{N}}c_{\hat{m},\hat{n}} P_{\hat{m}+m-\hat{M},\hat{n}+n-\hat{N}}\,,
\end{align}
where potential negative indices are treated according to \eqref{eqn:negative_P}.
Thus
\begin{align}\label{eqn:norm_lambda}
    \left\Vert\Lambda\right\Vert= \max\left\{\frac{\Vert \Lambda P_{m,n}\Vert}{\Vert P_{m,n}\Vert}:0\leq m\leq\hat{M}\ \land\ 0\leq n\leq\hat{N}\right\}\,.
\end{align}

In the end, we derive a simple bound on $\Vert \Lambda \Vert$. Since for $(m,n)\in\mathbb{Z}^2$ it holds $\Vert P_{m,n}\Vert\leq\Vert P_{|m|,|n|}\Vert$, we have
\begin{align}
    \left\Vert \Lambda P_{m,n}\right\Vert &\leq \sum_{\hat{m}=0}^{2\hat{M}}\sum_{\hat{n}=0}^{2\hat{N}}|c_{\hat{m},\hat{n}}| \left\Vert P_{\hat{m}+m-\hat{M},\hat{n}+n-\hat{N}}\right\Vert \leq \sum_{\hat{m}=0}^{2\hat{M}}\sum_{\hat{n}=0}^{2\hat{N}}|c_{\hat{m},\hat{n}}| \left\Vert P_{|\hat{m}+m-\hat{M}|,|\hat{n}+n-\hat{N}|}\right\Vert \\
    &= \sum_{\hat{m}=0}^{2\hat{M}}\sum_{\hat{n}=0}^{2\hat{N}}|c_{\hat{m},\hat{n}}| \rho_\tau^{2|\hat{m}+m-\hat{M}|+1} \rho_x^{2|\hat{n}+n-\hat{N}|+1}\\
    &\leq \rho_\tau^{2m+1} \rho_x^{2n+1} \sum_{\hat{m}=0}^{2\hat{M}}\sum_{\hat{n}=0}^{2\hat{N}}|c_{\hat{m},\hat{n}}| \rho_\tau^{2|\hat{m}-\hat{M}|} \rho_x^{2|\hat{n}-\hat{N}|} \,.
\end{align}
It leads to
\begin{align}
     \left\Vert \Lambda\right\Vert \leq \sum_{\hat{m}=0}^{2\hat{M}}\sum_{\hat{n}=0}^{2\hat{N}}\left|c_{\hat{m},\hat{n}} \right| \rho_\tau^{2|\hat{m}-\hat{M}|} \rho_x^{2|\hat{n}-\hat{N}|}\,.
\end{align}

\subsection{Bounds for \texorpdfstring{$H_0$}{H0}}
\label{sec:BoundsOnH0}
First, let us divide the set of all basis vectors $P_{m,n}$ of $X$ into two families:
\begin{align*}
    \mathcal{Y}&=\left\{P_{m,n}:\ m\geq\hat{M} \ \land\ n\geq \hat{N}\right\}\,,
    \\
    \mathcal{Z}&=\left\{P_{m,n}:\ m<\hat{M} \ \lor\ n< \hat{N}\right\}\,.
\end{align*}
The set $\mathcal{Y}$ consists of vectors such that after acting with $\Lambda$, the resulting linear combination has coefficients given exactly by \eqref{eq:uvP_decomposition}. In particular, for every $P_{m,n}\in\mathcal{Y}$ the value of $\Vert \Lambda P_{m,n}\Vert/\Vert P_{m,n}\Vert$ is the same. Remaining basis vectors belong to $\mathcal{Z}$. For them, certain basis vectors in the sum \eqref{eq:uvP_decomposition} match.

Now, thanks to Lemmas \ref{lem:band1} and \ref{lem:band2} we are able to find for a fixed $\sigma$ such values $\eta$ and $\chi$ that $\Vert L^{-1} P_{m,n}\Vert/\Vert P_{m,n}\Vert<\frac{\sigma}{3}\Vert \Lambda\Vert^{-1}$ outside of the band
\begin{align}
    \mathcal{B}_0=\left\{P_{m,n}:\ n\leq\frac{2p+1}{q}m+\eta \ \land\ m\leq\frac{q}{2p+1}n+\chi \right\}\,.
\end{align}
Then, let us consider the set
\begin{align}
\mathcal{D}=\left\{P_{m,n}:\ \mathcal{R}_{\hat{M},\hat{N}}^{(m,n)}\cap \mathcal{B}_0=\emptyset\right\}\,,
\end{align}
where $\mathcal{R}_{\hat{M},\hat{N}}^{(m,n)}$ has been introduced in \eqref{eqn:Rcal}. We observe that for every $P_{m,n}$ belonging to $\mathcal{D}$, $\Lambda P_{m,n}$ stays outside of the band $\mathcal{B}_0$, so such basis vector satisfies 
\begin{align*}
   \Vert H_0P_{m,n}\Vert= \left\Vert 3 L^{-1}\Lambda P_{m,n}\right\Vert \leq \sigma \Vert \Lambda\Vert^{-1} \left\Vert\Lambda P_{m,n}\right\Vert\leq \sigma \left\Vert P_{m,n}\right\Vert\,.
\end{align*}
Simple geometrical considerations lead to the conclusion that $\mathcal{D}$ is a complement of the inflated band
\begin{align}
    \mathcal{B}=\left\{P_{m,n}:\ n\leq\frac{2p+1}{q}(m+\hat{M})+\hat{N}+\eta \ \land\ m\leq\frac{q}{2p+1}(n+\hat{N})+\hat{M}+\chi \right\}\,.
\end{align}
Hence, for $P_{m,n}$ outside of $\mathcal{B}$ we have $\Vert H_0P_{m,n}\Vert/\Vert P_{m,n}\Vert<\sigma$.

Let us point out that $\mathcal{B}\cap\mathcal{Z}$ is a finite set, so estimating $\Vert H_0P_{m,n}\Vert/\Vert P_{m,n}\Vert<\sigma$ for its elements requires only a finite number of calculations that can be done with the help of a computer. The remaining part $\mathcal{B}\cap\mathcal{Y}$ can be dealt with by the means of the following lemma.
\begin{lemma}\label{lem:decreasing_g}
    Let $P_{m,n}\in \mathcal{Y}$. Then
    \begin{align*}
        \frac{\Vert H_0 P_{m,n}\Vert}{\Vert P_{m,n}\Vert}\leq g(m,n)\,,
    \end{align*}
    where 
    \begin{align*}
        g(m,n)=3 \sum_{\hat{m}=0}^{2\hat{M}}\sum_{\hat{n}=0}^{2\hat{N}}|c_{\hat{m},\hat{n}}|\frac{\left\Vert L^{-1} P_{\hat{m}+m-\hat{M},\hat{n}+n-\hat{N}}\right\Vert}{\left\Vert P_{\hat{m}+m-\hat{M},\hat{n}+n-\hat{N}}\right\Vert} \rho_\tau^{2(\hat{m}-\hat{M})} \rho_x^{2(\hat{n}-\hat{N})}\,.
    \end{align*}
    Additionally, if condition $(\ast)$ holds for the triple $(p,q,\rho_x)$, the function $g(m+qk,n+(2p+1)k)$ is decreasing in $k\in\mathbb{N}$ for every fixed $(m,n)\in\mathbb{N}^2$,.
\end{lemma}
\begin{proof}
    Since for $P_{m,n}\in \mathcal{Y}$ it holds
\begin{align*}
    \Lambda P_{m,n}\subset \textup{span}\, \mathcal{R}_{\hat{M},\hat{N}}^{(m,n)}= \textup{span}\{P_{\hat{m},\hat{n}}:\ m-\hat{M}\leq \hat{m}\leq m+\hat{M}\ \land\ n-\hat{N}\leq \hat{n}\leq n+\hat{N}\}\,,
\end{align*}
    we also have
    \begin{align*}
         \left\Vert H_0 P_{m,n} \right\Vert&=\left\Vert 3L^{-1}(\Lambda P_{m,n}) \right\Vert \leq
         3\sum_{\hat{m}=0}^{2\hat{M}}\sum_{\hat{n}=0}^{2\hat{N}}|c_{\hat{m},\hat{n}}| \left\Vert L^{-1}P_{\hat{m}+m-\hat{M},\hat{n}+n-\hat{N}}\right\Vert\\
         &=3\sum_{\hat{m}=0}^{2\hat{M}}\sum_{\hat{n}=0}^{2\hat{N}}|c_{\hat{m},\hat{n}}|\frac{\left\Vert L^{-1} P_{\hat{m}+m-\hat{M},\hat{n}+n-\hat{N}}\right\Vert}{\left\Vert P_{\hat{m}+m-\hat{M},\hat{n}+n-\hat{N}}\right\Vert}\left\Vert P_{\hat{m}+m-\hat{M},\hat{n}+n-\hat{N}}\right\Vert\\
         &=3 \left\Vert P_{m,n}\right\Vert\sum_{\hat{m}=0}^{2\hat{M}}\sum_{\hat{n}=0}^{2\hat{N}}|c_{\hat{m},\hat{n}}|\frac{\left\Vert L^{-1} P_{\hat{m}+m-\hat{M},\hat{n}+n-\hat{N}}\right\Vert}{\left\Vert P_{\hat{m}+m-\hat{M},\hat{n}+n-\hat{N}}\right\Vert} \rho_\tau^{2(\hat{m}-\hat{M})} \rho_x^{2(\hat{n}-\hat{N})}\\
         &=g(m,n) \left\Vert P_{m,n}\right\Vert\,.
    \end{align*}
    This gives us the first part of the result. The fact that $g(m+qk,n+(2p+1)k)$ is a decreasing function of $k\in\mathbb{N}$ is then a direct consequence of Lemma \ref{lem:f_decreasing}, since
    \begin{align*}
        g(m+qk,n+(2p+1)k)=3 \left\Vert P_{m,n}\right\Vert\sum_{\hat{m}=0}^{2\hat{M}}\sum_{\hat{n}=0}^{2\hat{N}}|c_{\hat{m},\hat{n}}|f_{m+\hat{m}-\hat{M},n+\hat{n}-\hat{N}}(k) \rho_\tau^{2(\hat{m}-\hat{M})} \rho_x^{2(\hat{n}-\hat{N})}\,.
    \end{align*}
\end{proof}

To sum up this subsection, we have shown that outside of a certain band $\mathcal{B}$ it holds $\Vert H_0P_{m,n}\Vert/\Vert P_{m,n}\Vert<\sigma$, hence, it is sufficient to choose $A P_{m,n}=P_{m,n}$ there. On the other hand, the interior of $\mathcal{B}$ can be divided into two sets: a finite $\mathcal{B}\cap\mathcal{Z}$ and an infinite $\mathcal{B}\cap\mathcal{Y}$. Lemma~\ref{lem:decreasing_g} tells us that norms of the elements of the latter can be bounded by the function $g$. In the next Section we show how one can use $g$ to construct a finite family $\mathcal{S}_{H_0}\subset\mathcal{B}$ such that for all $P_{m,n}\in\mathcal{B}\setminus \mathcal{S}_{H_0}$ it holds $\Vert H_0P_{m,n}\Vert/\Vert P_{m,n}\Vert<\sigma$.

\section{Computer-assisted verification}
\label{sec:ComputeAssistedVerification}

\subsection{Simplification of \texorpdfstring{$\Lambda$}{Lambda}}

In practice, the coefficients $c_{\hat{m},\hat{n}}$ in \eqref{eq:uvP_decomposition} often have large numerators and
denominators, making direct computations with $\Lambda$ costly. We therefore introduce a simpler approximate operator
$\tilde{\Lambda}$, defined on the basis elements by
\begin{equation}
    \tilde{\Lambda} P_{m,n}=\sum_{\hat{m}=0}^{2\hat{M}}\sum_{\hat{n}=0}^{2\hat{N}}\tilde{c}_{\hat{m},\hat{n}}P_{\hat{m}+m-\hat{M},\hat{n}+n-\hat{N}}\,,
\end{equation}
where $\tilde{c}_{\hat{m},\hat{n}}$ is a fixed table of coefficients. Negative indices are treated according to \eqref{eqn:negative_P}.

Using $\tilde{\Lambda}$, we rewrite $H$ as
\begin{equation}\label{eq:H_Lambda_tilde}
    H=I-\left(I+3 L^{-1}\circ \tilde{\Lambda}\right)\circ A+3L^{-1}\circ\left(\tilde{\Lambda}-\Lambda\right)\circ A\,.
\end{equation}
The last term accounts for the error introduced by replacing $\Lambda$ with $\tilde{\Lambda}$. We bound this term explicitly using
\begin{equation}
    \left(\tilde{\Lambda}-\Lambda\right)P_{m,n}=\sum_{\hat{m}=0}^{2\hat{M}}\sum_{\hat{n}=0}^{2\hat{N}}\left(\tilde{c}_{\hat{m},\hat{n}} - c_{\hat{m},\hat{n}} \right)P_{\hat{m}+m-\hat{M},\hat{n}+n-\hat{N}}\,,
\end{equation}
which yields
\begin{align}
    \left\Vert \left(\tilde{\Lambda}-\Lambda\right)P_{m,n} \right\Vert & \leq 
    \sum_{\hat{m}=0}^{2\hat{M}}\sum_{\hat{n}=0}^{2\hat{N}}\left|\tilde{c}_{\hat{m},\hat{n}} - c_{\hat{m},\hat{n}} \right|\left\Vert P_{\hat{m}+m-\hat{M},\hat{n}+n-\hat{N}}\right\Vert
    \nonumber
    \\
    &=\sum_{\hat{m}=0}^{2\hat{M}}\sum_{\hat{n}=0}^{2\hat{N}}\left|\tilde{c}_{\hat{m},\hat{n}} - c_{\hat{m},\hat{n}} \right|\frac{\left\Vert P_{\hat{m}+m-\hat{M},\hat{n}+n-\hat{N}}\right\Vert}{\Vert P_{m,n}\Vert}\Vert P_{m,n}\Vert
    \nonumber
    \\
    &\leq \Vert P_{m,n}\Vert \sum_{\hat{m}=0}^{2\hat{M}}\sum_{\hat{n}=0}^{2\hat{N}}\left|\tilde{c}_{\hat{m},\hat{n}} - c_{\hat{m},\hat{n}} \right| \rho_\tau^{2|\hat{m}+m-\hat{M}|-2m} \rho_x^{2|\hat{n}+n-\hat{N}|-2n}
    \nonumber
    \\
    \label{eq:26.10.12_01}
    & \leq \Vert P_{m,n}\Vert \sum_{\hat{m}=0}^{2\hat{M}}\sum_{\hat{n}=0}^{2\hat{N}}\left|\tilde{c}_{\hat{m},\hat{n}} - c_{\hat{m},\hat{n}} \right| \rho_\tau^{2|\hat{m}-\hat{M}|} \rho_x^{2|\hat{n}-\hat{N}|}\,.
\end{align}

The preceding estimates remain valid with $\Lambda$ replaced by $\tilde{\Lambda}$ and the truncation indices $\hat{M}$, $\hat{N}$ and size-dependent quantities $\eta$, $\chi$ adjusted to its reduced support. We use $\tilde{\Lambda}$ and the corresponding quantities in all subsequent formulae. The only additional contribution to be controlled is the error term in $H$ estimated above \eqref{eq:26.10.12_01}, in order to preserve the intended bound on $H$.
The construction of $\tilde{\Lambda}$, including the choice of its simplified coefficients and truncation parameters, is described in the next subsection.

\subsection{Choice of parameters}
\label{sec:ChoiceOfParameters}

The choice $p=11$ and $q=8$, which yields the frequency $\Omega$ stated in Theorem~\ref{thm:main}, is motivated by the numerical observation of three solutions at this frequency: one lies on the trunk of the bifurcation diagram, while the other two lie on a branch (see Figure~2 in \cite{Ficek.2025a}). This branch is also the first to appear as the truncation order of the Galerkin systems is increased \cite{Ficek.2025a}. Moreover, the relatively small values of $p$ and $q$ help control the arithmetic complexity of the computations.

The weights were selected through numerical experiments and were not finely tuned. In particular, the spatial weight $\rho_x>1$ was chosen to ensure sufficiently rapid decay of $\lVert H_0 P_{m,n}\rVert/\lVert P_{m,n}\rVert$ along the line $n=\Omega\,m$. We found that $\rho_x=11/10$ provides adequate decay. For the temporal weight, we adopted the value used in \cite{ficek2025new}.

The cut-off parameter was fixed at $\sigma=9/10$. This value was chosen conservatively, and no attempt was made to optimize it.

Further computational parameters are the truncation orders of the approximate solutions $v^{(i)}$. In view of the form of the linear operator $L$ in \eqref{eq:26.10.08_02}, it was important to retain more spatial modes than temporal modes. Inspection of the decay of the expansion coefficients suggested the choice $N=2M$. The specific truncation orders for the three approximate solutions were then selected by examining their respective residuals.

For each approximate solution $v$, we construct $\tilde{\Lambda}$ from a truncated expansion of $v^2$ with simplified rational coefficients. We first compute the coefficients of $v^2$ from those of $v$ and evaluate them using floating-point arithmetic with a precision of $64$ decimal digits. We then rationalize these coefficients with a relative tolerance of $10^{-10}$ to obtain simpler rational values. Finally, we truncate the resulting expansion at $(2M,2N)=(20,40)$ modes and $(2M,2N)=(24,48)$ modes for first and both second and third solutions respecitvely, retaining only the modes within these cut-offs. The resulting rational coefficients define $\tilde{\Lambda}$.
The same cut-offs are used for all three solutions. For the first solution, the expansion of $v^2$ already lies within these cut-offs, so no modes are removed: $\tilde{\Lambda}$ is obtained solely by simplifying the coefficients of $\Lambda$. This choice provides sufficient computational savings for all three cases, and a smaller cut-off for the first solution was not needed. The errors introduced by coefficient simplification and, where applicable, truncation remain sufficiently small for the required bounds to hold, see \eqref{eq:26.10.14_04}.

\subsection{Introducing relevant sets}

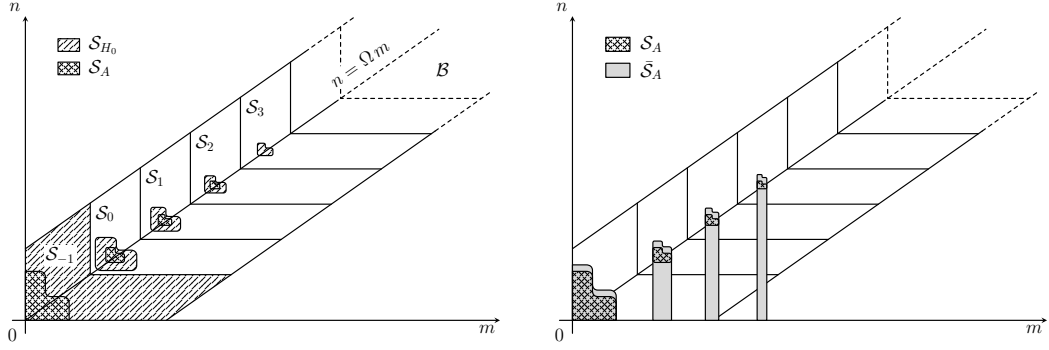
\begin{figure}[tbp]
  \centering
  \begin{minipage}{0.48\linewidth}
    \centering
    \resizebox{\linewidth}{!}{\begin{tikzpicture}[regions diagram]
\RegionsGeometry
\newcommand{\LeftHatchLabel}{\mathcal{S}_{H_0}}
\newcommand{\CrossHatchLabel}{\mathcal{S}_{A}}
\def\Sminusone{
    (0,0)
    --
    (0,\upperintercept)
    --
    (\cornerX,{\cornerY+\Lheight})
    --
    (\cornerX,\cornerY)
    --
    ({\cornerX+\Lwidth},\cornerY)
    --
    (\lowerxzero,0)
    --
    cycle
}
\path[left hatch] \Sminusone;
\RegionsBoundaries
\path[right hatch] \originblob;
\draw \originblob;
\def\q{0.72}
\def\baseBlobScale{0.90}
\def\innerRatio{0.45}
\foreach \k/\factor in {
    0/1,
    1/0.72,
    2/0.5184,
    3/0.373248
}{
    \pgfmathsetmacro{\cx}{\blobx+\k*\dx}
    \pgfmathsetmacro{\cy}{\bloby+\k*\dy}
    \pgfmathsetmacro{\sc}{\baseBlobScale*\factor}
    \begin{scope}[
        shift={(\cx,\cy)},
        scale=\sc
    ]
        \path[left hatch] \blob;
        \draw \blob;
    \end{scope}
}
\foreach \k/\factor in {
    0/1,
    1/0.72,
    2/0.5184
}{
    \pgfmathsetmacro{\cx}{\blobx+\k*\dx}
    \pgfmathsetmacro{\cy}{\bloby+\k*\dy}
    \pgfmathsetmacro{\innersc}{
        \baseBlobScale*\factor*\innerRatio
    }
    \begin{scope}[
        shift={(\cx,\cy)},
        scale=\innersc
    ]
        \path[right hatch] \blob;
        \draw \blob;
    \end{scope}
}
\RegionsCenterLine
\node[
    rotate={atan(\slope)},
    anchor=south,
    inner sep=1.5pt,
    fill=white
]
at (9.15,{9.15*\slope+0.30})
{$n=\Omega\,m$};
\node[
    fill=white,
    inner sep=1.5pt
]
at (0.94,1.75)
{$\mathcal{S}_{-1}$};
\foreach \k in {0,...,3} {
    \pgfmathsetmacro{\labelx}{2.15+\k*\dx}
    \pgfmathsetmacro{\labely}{2.88+\k*\dy}
    \node[
        fill=white,
        inner sep=1.4pt
    ]
    at (\labelx,\labely)
    {$\mathcal{S}_{\k}$};
}
\node[
    fill=white,
    inner sep=1.5pt
]
at (11.25,6.75)
{$\mathcal{B}$};
\begin{scope}[xshift=4mm]
\coordinate (LegA) at (0.48,7.25);
\path[left hatch]
    (LegA)
    rectangle ++(0.58,0.34);
\draw
    (LegA)
    rectangle ++(0.58,0.34);
\node[
    anchor=west
]
at (1.18,7.42)
{$\LeftHatchLabel$};
\coordinate (LegB) at (0.48,6.62);
\path[left hatch]
    (LegB)
    rectangle ++(0.58,0.34);
\path[right hatch]
    (LegB)
    rectangle ++(0.58,0.34);
\draw
    (LegB)
    rectangle ++(0.58,0.34);
\node[
    anchor=west
]
at (1.18,6.79)
{$\CrossHatchLabel$};
\end{scope}
\end{tikzpicture}}%
  \end{minipage}\hfill
  \begin{minipage}{0.48\linewidth}
    \centering
    \resizebox{\linewidth}{!}{\begin{tikzpicture}[regions diagram]
\RegionsGeometry
\def\nExtension{0.175}
\RegionsBoundaries
\RegionsCenterLine
\newcommand{\extendedblob}[3]{
    ({#1-0.57*#3},0)
    --
    ({#1-0.57*#3},
     {#2+0.43*#3+\nExtension})
    .. controls
       ({#1-0.57*#3},
        {#2+0.50*#3+\nExtension})
       and
       ({#1-0.53*#3},
        {#2+0.55*#3+\nExtension})
       ..
       ({#1-0.48*#3},
        {#2+0.55*#3+\nExtension})
    --
    ({#1-0.09*#3},
     {#2+0.55*#3+\nExtension})
    .. controls
       ({#1-0.01*#3},
        {#2+0.55*#3+\nExtension})
       and
       ({#1+0.06*#3},
        {#2+0.51*#3+\nExtension})
       ..
       ({#1+0.06*#3},
        {#2+0.44*#3+\nExtension})
    --
    ({#1+0.06*#3},
     {#2+0.30*#3+\nExtension})
    .. controls
       ({#1+0.06*#3},
        {#2+0.22*#3+\nExtension})
       and
       ({#1+0.10*#3},
        {#2+0.17*#3+\nExtension})
       ..
       ({#1+0.18*#3},
        {#2+0.17*#3+\nExtension})
    --
    ({#1+0.50*#3},
     {#2+0.17*#3+\nExtension})
    .. controls
       ({#1+0.59*#3},
        {#2+0.17*#3+\nExtension})
       and
       ({#1+0.67*#3},
        {#2+0.11*#3+\nExtension})
       ..
       ({#1+0.67*#3},
        {#2+0.02*#3+\nExtension})
    --
    ({#1+0.67*#3},0)
    --
    cycle
}
\newcommand{\extendedoriginblob}{
    (0,0)
    --
    (1.18,0)
    --
    (1.18,{0.47+\nExtension})
    .. controls
       (1.18,{0.58+\nExtension})
       and
       (1.10,{0.64+\nExtension})
       ..
       (1.00,{0.64+\nExtension})
    --
    (0.69,{0.64+\nExtension})
    .. controls
       (0.59,{0.64+\nExtension})
       and
       (0.54,{0.70+\nExtension})
       ..
       (0.54,{0.80+\nExtension})
    --
    (0.54,{1.17+\nExtension})
    .. controls
       (0.54,{1.27+\nExtension})
       and
       (0.47,{1.32+\nExtension})
       ..
       (0.37,{1.32+\nExtension})
    --
    (0,{1.32+\nExtension})
    --
    (0,0)
    --
    cycle
}
\path[fill=extendedFill]
    \extendedoriginblob;
\draw
    \extendedoriginblob;
\path[left hatch]
    \originblob;
\path[right hatch]
    \originblob;
\draw
    \originblob;
\foreach \k/\sc in {0/0.405,1/0.2916,2/0.209952} {
    \pgfmathsetmacro{\cx}{\blobx+\k*\dx}
    \pgfmathsetmacro{\cy}{\bloby+\k*\dy}

    \path[fill=extendedFill]
        \extendedblob{\cx}{\cy}{\sc};
    \draw
        \extendedblob{\cx}{\cy}{\sc};
}
\foreach \k/\sc in {0/0.405,1/0.2916,2/0.209952} {
    \pgfmathsetmacro{\cx}{\blobx+\k*\dx}
    \pgfmathsetmacro{\cy}{\bloby+\k*\dy}

    \begin{scope}[
        shift={(\cx,\cy)},
        scale=\sc
    ]
        \path[left hatch] \blob;
        \path[right hatch] \blob;
        \draw \blob;
    \end{scope}
}
\begin{scope}[xshift=4mm]
\coordinate (LegCross) at (0.50,7.25);
\path[left hatch]
    (LegCross) rectangle ++(0.65,0.38);
\path[right hatch]
    (LegCross) rectangle ++(0.65,0.38);
\draw
    (LegCross) rectangle ++(0.65,0.38);
\node[
    anchor=west
]
at (1.30,7.44)
{$\mathcal{S}_{A}$};
\coordinate (LegGray) at (0.50,6.55);
\path[
    fill=extendedFill
]
    (LegGray) rectangle ++(0.65,0.38);
\draw
    (LegGray) rectangle ++(0.65,0.38);
\node[
    anchor=west
]
at (1.30,6.74)
{$\bar{\mathcal{S}}_{A}$};
\end{scope}
\end{tikzpicture}}%
  \end{minipage}
  \caption{Illustration of sets used in the explicit computations, represented on the $(m,n)$-plane. \emph{Left:} The inflated band $\mathcal{B}$, the auxiliary sets $\setS_k$, and the construction of $\setS_{H_0}$, which uses the periodic structure. The set $\setS_A$ consists of the selected basis elements satisfying $\lVert H_0P_{m,n}\rVert/\lVert P_{m,n}\rVert\geq\sigma$, on which the non-trivial action of $A$ is applied. \emph{Right:} Construction of the extended set $\bar{\setS}_A\supseteq\setS_A$, which supplies the additional basis elements used to construct the matrix $\mathcal{A}$. The resulting operator $A$ acts as the identity outside $\setS_A$.}
  \label{fig:sets}
\end{figure}

We now define the sets used in the explicit computations, illustrated in Figure~\ref{fig:sets}. We begin with the auxiliary sets
\begin{equation}
	\setS_{-1} := \left\{P_{m,n}\,:\ (0\leq m\leq \hat{M})\ \lor\ (0\leq n\leq \hat{N})\right\}\cap \mathcal{B}\,,
\end{equation}
and
\begin{multline}
	\setS_{k} := \left\{P_{m,n}\,:\  (\hat{M}+kq< m )\ \land\ (\hat{N}+k(2p+1)< n) \ \land
	\right.
	\\
	\left.
	\left((m\leq \hat{M} +(k+1)q)\ \lor\ (n\leq \hat{N} + (k+1)(2p+1)) \right) \right\} \cap \mathcal{B}\,.
\end{multline}
where $k\in\mathbb{N}$. Note that $\mathcal{Z}\cap\mathcal{B} \subset \setS_{-1}$ while $\setS_{k}\subset \mathcal{Y}\cap\mathcal{B}$.

We next construct an additional auxiliary set, $\setS_g$, by recursively selecting indices at which $g\geq\sigma$, using the periodic structure described above. First, we evaluate $g(m,n)$ for each $P_{m,n}\in\setS_0$ and include in $\setS_g$ the basis elements for which $g(m,n)\geq\sigma$. We then translate the selected
indices by
\begin{equation}
	(m,n)\longmapsto(m+q,n+2p+1)\,,
\end{equation}
obtaining candidate indices in $\setS_1$. We evaluate $g$ at these indices and add the corresponding basis elements to $\setS_g$ whenever $g\geq\sigma$. At each subsequent step, we apply the same translation to the indices retained at the preceding step and repeat the test. The procedure stops when no indices satisfy $g\geq\sigma$. The decay property of $g$ is ensured by the fact that the triple $(p,q,\rho_x)=(11,8,\frac{11}{10})$ satisfies the condition $(\ast)$, as we prove in Appendix~\ref{sec:app_ast}.

We next combine $\setS_g$ and $\setS_{-1}$ to define the set
$\setS_{H_0}$:
\begin{equation}
	\setS_{H_{0}} := \setS_{-1} \cup \setS_{g}\,,
\end{equation}
For every $P_{m,n}\in\setS_{H_0}$, we explicitly compute $\lVert H_0P_{m,n}\rVert/\lVert P_{m,n}\rVert.$ These values determine the set $\setS_A$ through
\begin{equation}
	\setS_{A} := \left\{P_{m,n}\,:\ P_{m,n}\in\setS_{H_{0}}\ \land\ \frac{\left\|H_{0}P_{m,n}\right\|}{\left\|P_{m,n}\right\|}\geq\sigma \right\}\,,
\end{equation}
The set $\setS_A$ selects the basis elements on which we introduce a nontrivial correction through $A$ to obtain the required operator norm bound below $1$.

Finally, we define the extended set $\bar{\setS}_A$ by
\begin{equation}
	\bar{\setS}_{A} := \bigcup_{P_{m,n}\in \setS_{A}}\left\{P_{m,k}:\ k=0,1,\ldots,n+1\right\}\,.
\end{equation}
 
The sets used in the explicit computations and their relations are illustrated in Figure~\ref{fig:sets}. The left panel shows the construction of $\setS_{H_0}$ and the selection of $\setS_A$, while the right panel illustrates the extension of $\setS_A$ to $\bar{\setS}_A$ used to construct $\mathcal{A}$, which we discuss below.

\subsection{Construction of \texorpdfstring{$A$}{A}}

The finite-dimensional space spanned by $\bar{\setS}_A$ is used to construct the nontrivial action of $A$ on $\setS_A$. Thus, $\setS_A$ specifies where the correction is applied, while $\bar{\setS}_A$ supplies the additional modes needed for its construction. On basis elements outside $\setS_A$, $A$ acts as the identity. The explicit construction of $A$ is described below.

We first define the operator $\mathcal{Q}$ by
\begin{equation}
	\mathcal{Q} P_{m,n} = \left(I + 3L^{-1}\circ\tilde{\Lambda}\right)P_{m,n}\,.
\end{equation}
Restricting its domain to $\operatorname{span}\bar{\setS}_A$ and projecting its image onto the same subspace yields a square matrix. We convert its entries to floating-point numbers with a precision of $16$ decimal digits and compute its numerical inverse. We then set all entries of the inverse with absolute value less than $10^{-10}$ to zero and rationalize the remaining entries with a relative tolerance of $10^{-20}$. The resulting matrix defines an operator $\mathcal{A}$ on $\operatorname{span}\bar{\setS}_A$.

Finally, we define $A$ on the basis elements by, c.f.~Eq.~\eqref{eqn:operatorA},
\begin{equation}\label{eqn:opA}
	A P_{m,n}
	=
	\begin{cases}
		\mathcal{A}P_{m,n}, & \text{if } P_{m,n}\in\setS_A\,,
		\\
		P_{m,n}, & \text{if } P_{m,n}\notin\setS_A\,.
	\end{cases}
\end{equation}

\subsection{Rigorous evaluation of the norms}

For each of the three cases, the verification uses the supplied rational Fourier coefficients of $v$, the simplified coefficients defining $\tilde{\Lambda}$, and the rational matrix $\mathcal{A}$. All computations in this verification stage are performed in exact rational arithmetic.

We first compute $\lVert v\rVert$ and the Fourier coefficients of $v^2$ and $v^3$ from those of $v$. These quantities allow us to evaluate the residual $\mathcal{N}(0)=-L^{-1}v^{3}-v$ and its norm. We then construct $\tilde{\Lambda}$ from the supplied simplified coefficients of $v^2$ and evaluate its norm. Comparing these coefficients with the corresponding coefficients of $\Lambda$, while accounting for all modes omitted from $\tilde{\Lambda}$, gives a rigorous bound on $\lVert\Lambda-\tilde{\Lambda}\rVert$.

The main part of the verification is the estimation of $\lVert H\rVert$. We use $\lVert\tilde{\Lambda}\rVert$ to determine the band $\mathcal{B}_0$ and then construct its extension $\mathcal{B}$, taking into account the reduced support of $\tilde{\Lambda}$. The auxiliary function $g$ is evaluated according to the recursive procedure described above to determine $\setS_g$ and hence $\setS_{H_0}$. For every $P_{m,n}\in\setS_{H_0}$, we compute $\lVert H_0 P_{m,n}\rVert/\lVert P_{m,n}\rVert$ and select $\setS_A$ using the threshold $\sigma$.

The supplied matrix $\mathcal{A}$, indexed by $\bar{\setS}_A$, defines $A$ through the construction given in Eq.~\eqref{eqn:opA}.  Using the explicit form of $\mathcal{A}$, we compute $\lVert H_0 P_{m,n}\rVert/\lVert P_{m,n}\rVert$ for every $P_{m,n}\in\setS_A$. The analytical estimates established above bound the corresponding ratios outside $\setS_A$ by $\sigma$.

We also compute $\lVert A\rVert$ from its action on $\setS_A$ and its identity action on the remaining basis elements. For the parameters considered here (see Section~\ref{sec:ChoiceOfParameters}) the norm of $L^{-1}$ is attained on $P_{3,9}$ and equals
\begin{align}
    \left\Vert L^{-1}\right\Vert = \frac{1720890135138878791416128}{1850768595563353538700837} = 0.929824\ldots\,.
\end{align}
Combining these estimates with the error bound in \eqref{eq:26.10.12_01} gives
\begin{equation}
    \label{eq:26.10.14_04}
    \lVert H\rVert\leq\max\left\{\max_{P_{m,n}\in\setS_A}\frac{\lVert H P_{m,n}\rVert}
             {\lVert P_{m,n}\rVert}\,,\ \sigma \right\}
    + 3 \left\Vert L^{-1}\right\Vert \left\Vert A\right\Vert \left\Vert \tilde{\Lambda}-\Lambda\right\Vert\,.
\end{equation}
The resulting rational bounds are compared exactly with the values listed below.

\subsection{Verification of the theorem's inequalities and conclusions}

For each approximate solution $v^{(i)}$, $i=1,2,3$, we verify the prescribed norm bounds for $v$, $\mathcal{N}(0)$, $A$, and $H$, together with both inequalities \eqref{eq:25.05.21_05} in Theorem~\ref{thm:cap}.
All checks are performed in exact rational arithmetic. The bounds and parameter values used in these checks are listed below:
\begin{itemize}
\item 1st solution: $v=v^{(1)}$
	\begin{subequations}
		\label{eq:26.10.14_01}
		\noeqref{eq:A1,eq:B1,eq:C1,eq:D1,eq:E1,eq:F1}
		\begin{align}
			\left\Vert v\right\Vert &< \frac{11932577993}{2618142364} = 4.55765\ldots \,, \label{eq:A1}
			\\
			\left\Vert \mathcal{N}(0) \right\Vert &< \frac{1251640}{200388335065877} = 6.24607\ldots \times 10^{-9} \,, \label{eq:F1}
			\\
			\left\Vert A \right\Vert & < \frac{19906846310}{1894822989} = 10.5059\ldots \,, \label{eq:B1}
			\\
			\left\Vert H \right\Vert & < \frac{5959447941}{6621608675} = 0.900000020161\ldots \,, \label{eq:C1}
			\\
			\delta &= \frac{13081330001308133}{209065086978900000000000} = 6.25706\ldots \times 10^{-8}\,, \label{eq:E1}
			\\
			\Kzero &= \frac{6817223577}{7573215052} = 0.900175\ldots \,. \label{eq:D1}
		\end{align}
	\end{subequations}

\item 2nd solution: $v=v^{(2)}$
	\begin{subequations}
		\label{eq:26.10.14_02}
		\noeqref{eq:A2,eq:B2,eq:C2,eq:D2,eq:E2,eq:F2}
		\begin{align}
			\left\Vert v\right\Vert &< \frac{28738489637}{1295409482} = 22.1848\ldots \,, \label{eq:A2}
			\\
			\left\Vert \mathcal{N}(0) \right\Vert &< \frac{26953}{2679099598281738} = 1.00604\ldots\times 10^{-11}\,, \label{eq:F2}
			\\
			\left\Vert A \right\Vert & < \frac{70711511471}{1600927163} = 44.1690\ldots \,, \label{eq:B2}
			\\
			\left\Vert H \right\Vert & < \frac{11705819207}{11774624884} = 0.994156\ldots \,, \label{eq:C2}
			\\
			\delta &= \frac{263430000026343}{141217277216243600000000} = 1.86542\ldots\times 10^{-9}\,,\label{eq:E2}
			\\
			\Kzero &= \frac{11699300203}{11762738174} = 0.994606\ldots \,. \label{eq:D2}
		\end{align}
	\end{subequations}

\item 3rd solution: $v=v^{(3)}$
	\begin{subequations}
		\label{eq:26.10.14_03}
		\noeqref{eq:A3,eq:B3,eq:C3,eq:D3,eq:E3,eq:F3}
		\begin{align}
			\left\Vert v\right\Vert &< \frac{64508652992}{2911994189} = 22.1527\ldots \,, \label{eq:A3}
			\\
			\left\Vert \mathcal{N}(0) \right\Vert &< \frac{28312}{404403080500235} = 7.00093\ldots\times 10^{-11}\,, \label{eq:F3}
			\\
			\left\Vert A \right\Vert & < \frac{67312911313}{1520322491} = 44.2754\ldots \,, \label{eq:B3}
			\\
			\left\Vert H \right\Vert & < \frac{1881566659}{1897221322} = 0.991748\ldots \,, \label{eq:C3}
			\\
			\delta &= \frac{9160870000916087}{571871816392180000000000} = 1.60190\ldots\times 10^{-8} \,, \label{eq:E3}
			\\
			\Kzero &= \frac{1218606273}{11267850935} = 0.99562\ldots \,. \label{eq:D3}
		\end{align}
	\end{subequations}
\end{itemize}

Applying the fixed-point argument described in Section~\ref{sec:FixedPointFormulation}, we conclude that, for each $i=1,2,3$, there exists an exact solution $w^{(i)}$ to Eq.~\eqref{eq:conformal} satisfying
\begin{equation}
	 \lVert w^{(i)}-v^{(i)}\rVert<\varepsilon^{(i)}\,,
\end{equation}
where the error bounds, computed from the certified quantities above, are
\begin{equation}
	\label{eq:26.10.08_01}
	\varepsilon^{(1)}= 6.57361\ldots\,\times 10^{-7} \,,
	\quad
	\varepsilon^{(2)}= 8.23940\ldots\,\times 10^{-8} \,,
	\quad
	\varepsilon^{(3)}= 7.09252\ldots\,\times 10^{-7} \,.
\end{equation}

Finally, for every pair of distinct indices $i,j\in\{1,2,3\}$, we compute $\lVert v^{(i)}\pm v^{(j)}\rVert$ for both choices of sign. Combining these values with the error bounds $\varepsilon^{(i)}$ and $\varepsilon^{(j)}$, we use the triangle inequality to verify \eqref{eq:u-u0} and establish the required separation of the exact solutions.

This completes the proof of Theorem~\ref{thm:cap}.

\appendix

\section{Mathematica implementation and accompanying data}

\subsection{Description of the script and data files}

The supplementary materials contain two \Mathematica{} \cite{Mathematica.V15} scripts, \filename{Proof.m} and \filename{Distance.m}, together with three data sets corresponding to the approximate solutions $v^{(i)}$, $i=1,2,3$. The script \filename{Proof.m} verifies the bounds and inequalities required by Theorem~\ref{thm:cap}, while \filename{Distance.m} verifies the separation of the resulting solutions. The same \filename{Proof.m} script is used for all three cases, with the appropriate data set selected before execution.

The data sets are stored in the subdirectories \filename{Data/1st}, \filename{Data/2nd}, and \filename{Data/3rd}. The repository has the following structure:
\begin{ShellOutput}
.
|-- Proof.m
|-- Distance.m
|-- Data/
|   |-- 1st/
|   |   |-- vHat.m
|   |   |-- aCal.m
|   |   `-- vSquaredTilde.m
|   |-- 2nd/
|   |   |-- vHat.m
|   |   |-- aCal.m
|   |   `-- vSquaredTilde.m
|   `-- 3rd/
|       |-- vHat.m
|       |-- aCal.m
|       `-- vSquaredTilde.m
`-- Logs/
    |-- Proof_1st.log
    |-- Proof_2nd.log
    `-- Proof_3rd.log
\end{ShellOutput}
All files with the extension \filename{.m} are human-readable plain-text files containing \Mathematica{} code or data. Each data set supplies the rational coefficients and matrix entries used in the corresponding verification.

The file \filename{vHat.m} contains the Fourier coefficients of the approximate solution. Suppressing the case index $i$, we write
\begin{equation}
	v = \sum_{m=0}^{M-1}\sum_{n=0}^{N-1} \hat{v}_{m,n}P_{m,n}(\tau,x)\,.
\end{equation}
The coefficients are stored as a matrix, represented by a nested list in \Mathematica{}. The rows correspond to $m=0,\ldots,M-1$, and the columns to $n=0,\ldots,N-1$. Thus, because \Mathematica{} lists are indexed from $1$, the entry in row $m+1$ and column $n+1$ is $\hat{v}_{m,n}$.

The file \filename{vSquaredTilde.m} contains the simplified and truncated coefficients of $v^2$ used to construct $\tilde{\Lambda}$. In the script, the corresponding coefficient table is stored in \texttt{cTildeTab}. The coefficients defining the original operator $\Lambda$ are computed from $v^2$ and stored in \texttt{cTab}. Both tables are used when estimating $\lVert\Lambda-\tilde{\Lambda}\rVert$.

The file \filename{aCal.m} contains a list with two elements. The first is the rational matrix representing $\mathcal{A}$, loaded into the variable \texttt{aCal}. The second is an ordered list of pairs $(m,n)$ indexing the basis elements in $\bar{\setS}_A$, loaded into \texttt{setABar}. This ordering determines both the rows and columns of the matrix. More precisely, if the $r$th pair in \texttt{setABar} is $(m_r,n_r)$ and the stored matrix is $(\mathcal{A}_{rs})_{r,s=1}^{K}$, then
\begin{equation}
	\mathcal{A}P_{m_s,n_s} = \sum_{r=1}^{K} \mathcal{A}_{rs}P_{m_r,n_r}\,, \qquad s=1,\ldots,K\,.
\end{equation}
As described above, this matrix is used to define the action of $A$ on $\setS_A$, while $A$ acts as the identity on the remaining basis elements. The script verifies that $\setS_A\subseteq\bar{\setS}_A$ and terminates if this inclusion fails. This check ensures that the supplied matrix covers every basis element on which the correction is required.

For ease of navigation, \filename{Proof.m} is divided into three sections: \emph{Code}, \emph{Inputs}, and \emph{Computations}. The \emph{Code} section defines the routines used to manipulate the coefficient arrays, evaluate norms, construct the computational sets, and estimate the operators. The \emph{Inputs} section is further divided into \emph{Configuration}, \emph{Data}, and \emph{Parameters and bounds}. These specify, respectively, the computational settings, including the number of parallel kernels; the data set to be loaded; and the parameters and prescribed bounds for the selected case.

The \emph{Computations} section begins with the \emph{Initialisation} subsection, in which the prescribed number of \Mathematica{} kernels is launched. The verification then proceeds in the following order:
\begin{enumerate}
    \item Compute $\lVert v\rVert$, the coefficients of $v^2$ and $v^3$, evaluate the residual $\mathcal{N}(0)$, and evaluate its norm.
    
    \item Construct the coefficient tables defining $\Lambda$, evaluate $\lVert\tilde{\Lambda}\rVert$, and bound $\lVert\Lambda-\tilde{\Lambda}\rVert$.

    \item Determine the band $\mathcal{B}_0$ and its extension $\mathcal{B}$. Using the auxiliary function $g$, construct $\setS_g$ and then $\setS_{H_0}$.

    \item Compute $\lVert H_0P_{m,n}\rVert/\lVert P_{m,n}\rVert$ for every $P_{m,n}\in\setS_{H_0}$, select $\setS_A$, and verify that $\setS_A\subseteq\bar{\setS}_A$.

    \item Using the supplied matrix $\mathcal{A}$, compute $\lVert A\rVert$ and $\lVert H P_{m,n}\rVert/\lVert P_{m,n}\rVert$ for every $P_{m,n}\in\setS_A$.

    \item Combine the results to obtain the bound on $\lVert H\rVert$ according to \eqref{eq:26.10.14_04}.

    \item Verify the prescribed norm bounds and both inequalities in Theorem~\ref{thm:cap}.
\end{enumerate}
The section concludes with the \emph{Finalisation} subsection, which reports the runtime, memory usage, and closes the kernels launched during initialisation.

All verification computations are performed in exact rational arithmetic. Their mathematical justification is provided in the preceding sections.

The separate script \filename{Distance.m} loads the coefficients from \filename{vHat.m} for all three cases. It uses the corresponding prescribed bounds on $\lVert A\rVert$, already included in the script, to compute the error radii $\varepsilon^{(i)}$. For each pair $1\leq i<j\leq3$, it then evaluates $\lVert v^{(i)}\pm v^{(j)}\rVert$ for both choices of sign and verifies the separation inequalities \eqref{eq:u-u0} using these radii.

Instructions for executing both scripts and interpreting their output are given in the next subsection.

\subsection{Reproduction instructions and interpretation of the output}

The verification can be reproduced using \Mathematica{} 15.0.0. No external packages are required. The mathematical verification consists of arithmetic operations and comparisons on rational numbers, and all checks are performed exactly using the supplied data. Successful verification requires normal completion of the scripts, satisfaction of the prescribed norm bounds in \eqref{eq:26.10.14_01}-\eqref{eq:26.10.14_03}, and verification of both inequalities in Theorem~\ref{thm:cap}.

The scripts and data files are available at \url{https://github.com/mmal/\gitRepName}. The repository can be downloaded by selecting \texttt{Code > Download ZIP} on its GitHub page and extracting the archive. Alternatively, it can be cloned using
\begin{ShellInput}[commandchars=\\\{\}]
git clone https://github.com/mmal/\gitRepName.git
cd \gitRepName
\end{ShellInput}
All scripts should be executed from the repository's main directory, with the supplied directory structure preserved. The script \texttt{Proof.m} loads the data for one solution, whereas \texttt{Distance.m} loads the data for all three solutions.

Before running \texttt{Proof.m}, set \texttt{dirRoot} to the directory corresponding to the case being verified: \texttt{1st}, \texttt{2nd}, or \texttt{3rd}. Set
\texttt{bndNormV}, \texttt{bndNormN0}, \texttt{bndNormA}, \texttt{bndNormLinv}, and \texttt{bndNormH} to the prescribed bounds for $\lVert v\rVert$, $\lVert\mathcal{N}(0)\rVert$, $\lVert A\rVert$, $\lVert L^{-1}\rVert$, and $\lVert H\rVert$, respectively. The variables \texttt{cnstDelta} and \texttt{cnstK} must likewise be set to the values of $\delta$ and $K$ specified for that case. These values should be entered as exact rational numbers.

We recommend command-line execution, with each case evaluated in a fresh kernel. On both macOS and Linux, the first case
can be run using
\begin{ShellInput}
wolframscript -local -file Proof.m 2>&1 | tee Proof_1st.log
\end{ShellInput}
This command assumes that \texttt{wolframscript} is available on the command search path and uses the local \Mathematica{} 15.0.0 kernel. It displays the output while also saving it to \texttt{Proof\_1st\_run.log}. Repeat the procedure for the second and third cases after updating \texttt{dirRoot}, the prescribed bounds, and the constants, using distinct output filenames such as \texttt{Proof\_2nd\_run.log} and \texttt{Proof\_3rd\_run.log}.

The script \texttt{Proof.m} reports computed bounds and diagnostic quantities throughout its execution. The decisive outputs are the Boolean results of the norm comparisons and the two final inequalities in Theorem~\ref{thm:cap}. For each case, successful verification requires all six inequality checks to return \texttt{True}, as illustrated by the following excerpt from \filename{Proof_2nd.log}:

\begin{ShellOutput}[commandchars=\\\{\}]
#23  Fri 16 Oct 2026 08:38:51
|-- condition  normV <= bndNormV
`-- result     True

#24  Fri 16 Oct 2026 08:38:51
|-- condition  normN0 <= bndNormN0
`-- result     True

#25  Fri 16 Oct 2026 08:38:51
|-- condition  normA <= bndNormA
`-- result     True

#26  Fri 16 Oct 2026 08:38:51
|-- condition  normH <= bndNormH
`-- result     True

#27  Fri 16 Oct 2026 08:38:51
|-- condition  bndNormH + 6*bndNormLinv*bndNormV*bndNormA^2*\ensuremath{\delta} + 3*bndNormLinv*bndNormA^3*\ensuremath{\delta}^2 < K
`-- result     True

#28  Fri 16 Oct 2026 08:38:51
|-- condition  bndNormN0 < (1 - K)*\ensuremath{\delta}
`-- result     True
\end{ShellOutput}

The pairwise separation of the solutions is verified by
running
\begin{ShellInput}
wolframscript -local -file Distance.m 2>&1 | tee Distance.log
\end{ShellInput}
Successful completion produces
\begin{ShellOutput}
{{True, True}, {True, True}, {True, True}}
\end{ShellOutput}
Each pair of Boolean values corresponds to one pair of solutions and the two choices of sign in the separation inequalities. Thus, all six comparisons must return \texttt{True}.

Alternatively, the scripts may be evaluated interactively in \Mathematica{}. In this case, set the working directory to the repository's main directory, start with a fresh kernel, and evaluate all commands in their original order. The output should be saved explicitly. Command-line execution is recommended because it provides a straightforward way to preserve a complete execution log.

The approximate running times of \texttt{Proof.m} were 1127\,s, 95990\,s, and 52193\,s for the first, second, and third solutions, respectively. The corresponding peak memory requirements were $22.6245\,\mathrm{GiB}$, $185.053\,\mathrm{GiB}$, and $173.817\,\mathrm{GiB}$. The computations were performed on a workstation hosted by TU~Wien, equipped with four \href{https://www.intel.com/content/www/us/en/products/sku/84681/intel-xeon-processor-e78867-v3-45m-cache-2-50-ghz/specifications.html}{Intel Xeon E7-8867 v3} processors operating at $2.50\,\mathrm{GHz}$ and $2\,\mathrm{TiB}$ of RAM. Each run used $16$ \Mathematica{} kernels. The corresponding output files, \texttt{Proof\_1st.log}, \texttt{Proof\_2nd.log}, and \texttt{Proof\_3rd.log}, are included in the repository for comparison. Once its input data have been loaded, \texttt{Distance.m} completes almost immediately.

\section{Verification of condition \texorpdfstring{$(\ast)$}{(*)} for the chosen parameters}
\label{sec:app_ast}

\begin{lemma}
    The condition $(\ast)$ holds for the triple $(p,q,\rho_x)=(11,8,\frac{11}{10})$.
\end{lemma}
\begin{proof}
    Let us write expression \eqref{eqn:condition_f} as a polynomial in $\Delta$.
    \begin{align}
        \Theta(\Delta):=& d_3\Delta^3+d_2 \Delta^2+d_1 \Delta+d_0\,,
    \end{align}
    where
    \begin{align*}
        d_3&=(2p+1)\,,\\
        d_2&= 4(2p+1)(\nu_{k}+\nu_{k+1})+(2\nu_k+1)S_k- (2\nu_{k+1}+1)S_{k+1}\,,\\
        d_1&= 16(2p+1)\nu_k \nu_{k+1}+2(2\nu_k+1)(4\nu_{k+1}+1)S_{k}- 2(2\nu_{k+1}+1)(4\nu_k+1)S_{k+1}\,,\\
        d_0&= 8(2\nu_k+1)(2\nu_{k+1}+1)(\nu_{k+1}S_k-\nu_k S_{k+1})\,.
    \end{align*}
    
    First, let us assume that either $k$ or $n$ is nonzero, so $\nu_k\geq 1$. Then we can rewrite $d_0$ as
    \begin{align*}
        d_0 = 8(2\nu_k+1)(2\nu_{k+1}+1)\nu_{k+1}\nu_{k}\left(\frac{S_k}{\nu_k}- \frac{S_{k+1}}{\nu_{k+1}}\right)\,.
    \end{align*}
    Since $S_k/\nu_k$ is a mean of decreasing elements, it decreases with $k$ and $d_0>0$. For $d_2$ we have
    \begin{align*}
        d_2 &= 4(2p+1)(\nu_k+\nu_{k+1})-(2\nu_k+1)\sum_{j=\nu_k+1}^{\nu_k+2p+1}\rho_x^{-2j}-2(2p+1) \sum_{j=1}^{\nu_{k+1}}\rho_x^{-2j}\\
        &\geq4(2p+1)(\nu_k+\nu_{k+1})-(2p+1)(2\nu_k+1)-2(2p+1)\nu_{k+1}
        \\
        &=(2p+1)(2\nu_k+2\nu_{k+1}-1)>0\,.
    \end{align*}
    Now, we introduce $s_k=S_{k+1}-S_k$. Since $\rho_x=\frac{11}{10}$, $s_k<5$ for every $k\in\mathbb{N}$. Then for $p=11$ we have
    \begin{align*}
        d_1&=16\nu_k^2(2p+1-s_k)+\left[16(2p+1)(2p+1-s_k)-12s_k\right]\nu_k+4(2p+1) S_k-2(4p+3)s_k\\
        &\geq 288 \nu_k^2+6564\nu_k-470\,.
    \end{align*}
    As $\nu_k\geq 1$, $d_1>0$ and all of the coefficients of $\Theta(\Delta)$ are positive. It leads to expression \eqref{eqn:condition_f} being positive for $\Delta>0$.
    
    Now, let us assume that $k=n=0$, so we have (for $p=11$, $\rho_x=\frac{11}{10}$)
    \begin{align}\label{eq:Delta_kn0}
        \Theta(\Delta)=\left(23\Delta^2+\frac{251\,336}{121}\Delta-\frac{9400}{121}\right)\Delta\,.
    \end{align}
    Since then $\Delta=\frac{23}{8}(2m+1)\geq\frac{23}{8}$, polynomial \eqref{eq:Delta_kn0} is positive and so is expression \eqref{eqn:condition_f}.
\end{proof}

\medskip

\section*{Statements}
\noindent\textbf{Data availability.} The data supporting the findings of this study are provided in the paper and in the supplementary files available from the accompanying GitHub repository \url{https://github.com/mmal/\gitRepName}.

\medskip

\noindent\textbf{Conflict of interest.} The authors declare no conflicts of interest.

\medskip

\noindent\textbf{Use of AI.} The authors declare that the use of generative AI was limited to improving the grammar and formatting of the manuscript.

\medskip

\printbibliography

@unpublished{Maliborski.2016,
  author       = {Maliborski, Maciej},
  title = {{A toy model of geons–-cubic wave equation on AdS with axial symmetry}},
  year  = {2016},
  note   = {Unpublished manuscript},
}

@misc{Bizon.Personal2018,
  author       = {Bizoń, Piotr},
  howpublished = {Personal communication},
  year         = {2018},
}

@article{Rostworowski.2017,
  author    = {Andrzej Rostworowski},
  title     = {{Higher order perturbations of anti--de Sitter space and time-periodic solutions of vacuum Einstein equations}},
  journal   = {Physical Review D},
  volume    = {95},
  pages     = {124043},
  year      = {2017},
  doi       = {10.1103/PhysRevD.95.124043},
  eprint    = {1701.07804},
  archivePrefix = {arXiv},
  primaryClass  = {gr-qc},
  publisher = {American Physical Society}
}

@unpublished{Ficek.Future,
  author = {Ficek, Filip and Maliborski, Maciej},
  title  = {{New time-periodic solutions to the cubic conformal wave equation on the Einstein cylinder: II Axial symmetry}},
  note   = {Manuscript in preparation},
  year   = {2027}
}

@article{Ficek.2025a,
  title = {Complex Structure of Time-Periodic Solutions Decoded in {{Poincaré-Lindstedt}} Series: The Cubic Conformal Wave Equation on $\mathbb{S}^3$},
  shorttitle = {Complex Structure of Time-Periodic Solutions Decoded in {{Poincaré-Lindstedt}} Series},
  author = {Ficek, Filip and Maliborski, Maciej},
  year = {2025},
  journaltitle = {Physica D: Nonlinear Phenomena},
  shortjournal = {Physica D: Nonlinear Phenomena},
  volume = {481},
  eprint = {2504.10513},
  eprinttype = {arXiv},
  eprintclass = {math.AP},
  pages = {134864},
  issn = {01672789},
  doi = {10.1016/j.physd.2025.134864}
}

@article{Evnin.2020fys, 
year = {2020}, 
title = {{Breathing Modes, Quartic Nonlinearities and Effective Resonant Systems}}, 
author = {Evnin, Oleg}, 
journal = {Symmetry, Integrability and Geometry: Methods and Applications}, 
doi = {10.3842/sigma.2020.034}, 
eprint = {1912.07952}, 
  archivePrefix = {arXiv},
primaryClass = {math-ph}
}

@article{Bizoń.2014, 
year = {2014}, 
title = {{Is AdS stable?}}, 
author = {Bizoń, Piotr}, 
journal = {General Relativity and Gravitation}, 
issn = {0001-7701}, 
doi = {10.1007/s10714-014-1724-0}, 
eprint = {1312.5544},
  archivePrefix = {arXiv},
  primaryClass  = {gr-qc}, 
pages = {1724}, 
number = {5}, 
volume = {46}
}

@article{Bizoń.2017ndf, 
year = {2017}, 
title = {{Conformal Flow on $S^{3}$ and Weak Field Integrability in AdS${}_{4}$}}, 
author = {Bizoń, Piotr and Craps, Ben and Evnin, Oleg and Hunik, Dominika and Luyten, Vincent and Maliborski, Maciej}, 
journal = {Communications in Mathematical Physics}, 
issn = {0010-3616}, 
doi = {10.1007/s00220-017-2896-8}, 
eprint = {1608.07227}, 
archivePrefix = {arXiv},
primaryClass = {math.AP}, 
pages = {1179--1199}, 
number = {3}, 
volume = {353}
}

@article{Chatzikaleas.2020oh, 
year = {2020}, 
title = {{On time periodic solutions to the conformal cubic wave equation on the Einstein cylinder}}, 
author = {Chatzikaleas, Athanasios}, 
journal = {Journal of Mathematical Physics}, 
issn = {0022-2488}, 
doi = {10.1063/5.0026015}, 
eprint = {2004.11064}, 
  archivePrefix = {arXiv},
  primaryClass  = {math.AP}, 
pages = {111505}, 
number = {11}, 
volume = {61}
}

@article{Chatzikaleas.2024, 
year = {2024}, 
title = {{Nonlinear periodic waves on the Einstein cylinder}}, 
author = {Chatzikaleas, Athanasios and Smulevici, Jacques}, 
journal = {Analysis \& PDE}, 
issn = {2157-5045}, 
doi = {10.2140/apde.2024.17.2311}, 
eprint = {2201.05447}, 
  archivePrefix = {arXiv},
  primaryClass  = {math.AP}, 
pages = {2311--2378}, 
number = {7}, 
volume = {17}
}

@article{Silimbani.2024, 
year = {2024}, 
title = {{Massive Cantor families of periodic solutions of resonant Klein-Gordon equation on $\mathbb{S}^{3}$}}, 
author = {Silimbani, Diego}, 
journal = {arXiv}, 
doi = {10.48550/arxiv.2409.14845}, 
eprint = {2409.14845}, 
  archivePrefix = {arXiv},
  primaryClass  = {math.AP}
}

@article{Berti.2024, 
year = {2024}, 
title = {{Time-periodic solutions of completely resonant Klein–Gordon equations on $\mathbb{S}^{3}$}}, 
author = {Berti, Massimiliano and Langella, Beatrice and Silimbani, Diego}, 
journal = {Annales de l'Institut Henri Poincaré C, Analyse non linéaire}, 
issn = {0294-1449}, 
doi = {10.4171/aihpc/125}, 
eprint = {2308.05678}, 
  archivePrefix = {arXiv},
  primaryClass  = {math.AP}
}

@software{Mathematica.V15,
  author       = {{Wolfram Research, Inc.}},
  title        = {Mathematica},
  year         = {2026},
  version      = {15.0},
  organization = {Wolfram Research, Inc.},
  address      = {Champaign, Illinois},
  url          = {https://www.wolfram.com/mathematica/},
  shorthand    = {MMA},
}

@article{Ficek.2024A,
  title={{Periodic solutions for the 1D cubic wave equation with Dirichlet boundary conditions}},
  author={Ficek, Filip and Maliborski, Maciej},
  journal={Nonlinearity},
  volume={38},
  number={6},
  pages={065016},
  year={2025},
  publisher={IOP Publishing},
  doi={10.1088/1361-6544/add831},
  eprint = {2407.16507},
  archivePrefix = {arXiv},
  primaryClass  = {math.AP},
}

@misc{Ficek.2024B,
	author={Ficek, F. and Maliborski, M.},
	title={{Trees, trunks, and branches - bifurcation structure of time-periodic solutions to $u_{tt}-u_{xx}\pm u^3=0$}},
	year={2024},
    journal = {arXiv},
    doi = {10.48550/arxiv.2408.05158},
	eprint = {2408.05158},
      archivePrefix = {arXiv},
  primaryClass  = {math.AP},
}

@article{Arioli.2017, 
year = {2017}, 
title = {{Families of Periodic Solutions for Some Hamiltonian PDEs}}, 
author = {Arioli, Gianni and Koch, Hans}, 
journal = {SIAM Journal on Applied Dynamical Systems}, 
doi = {10.1137/16m1070177}, 
pages = {1--15}, 
number = {1}, 
volume = {16}
}

@book{berti2007nonlinear,
  title={{Nonlinear oscillations of Hamiltonian PDEs}},
  author={Berti, Massimiliano},
  volume={74},
  year={2007},
  publisher={Birkhäuser},
  address = {Boston},
  doi = {10.1007/978-0-8176-4681-3},
  issn = {1421-1750},
  series = {Progress in Nonlinear Differential Equations and Their Applications}
}

@book{arnold1988dynamical,
    title={{Dynamical systems III}},
    editor={Vladimir I. Arnold},
    author={Arnold, VI and Kozlov, VV and Neishtadt, AI},
    year={1988},
    series={{Encyclopaedia of Mathematical Sciences}},
    publisher={Springer Berlin, Heidelberg},
    isbn={978-3-662-02535-2},
    issn={0938-0396},
    doi={10.1007/978-3-662-02535-2}
}

@misc{ficek2025new,
  author        = {Filip Ficek and Maciej Maliborski},
  title         = {New class of time-periodic solutions to the 1D cubic wave equation},
  year          = {2025},
  eprint        = {2506.10839},
  archivePrefix = {arXiv},
  primaryClass  = {math.AP},
  doi           = {10.48550/arXiv.2506.10839},
}

@misc{Ficek2025timeperiodic,
  author        = {Filip Ficek},
  title         = {{Time-periodic solutions to the cubic wave equation: an elementary constructive approach}},
  year          = {2025},
  eprint        = {2510.20054},
  archivePrefix = {arXiv},
  primaryClass  = {math.AP},
  doi           = {10.48550/arXiv.2510.20054},
}

@article{Ficek.2026proc,
  author    = {Filip Ficek and Maciej Maliborski},
  title     = {{Time-periodic solutions of the conformally invariant wave equation on the Einstein cylinder}},
  journal   = {Journal of Physics: Conference Series},
  volume    = {3177},
  number    = {1},
  pages     = {012019},
  year      = {2026},
  doi       = {10.1088/1742-6596/3177/1/012019},
  publisher = {IOP Publishing},
  eprint = {2508.19717},
  archivePrefix = {arXiv},
  primaryClass  = {gr-qc},
}

@article{ficek2026nonlinear,
  author    = {Filip Ficek and Maciej Maliborski},
  title     = {{Nonlinear Oscillations of Strings and Beams}},
  journal   = {Progress of Theoretical and Experimental Physics},
  volume    = {2026},
  number    = {2},
  pages     = {021A01},
  year      = {2026},
  doi       = {10.1093/ptep/ptag002},
  eprint    = {2508.19091},
  archivePrefix = {arXiv},
  primaryClass  = {math-ph},
  publisher = {Oxford University Press}
}

\end{document}